\documentclass[12pt]{article}

\usepackage{amssymb,amsmath,amsthm,ucs}
  \usepackage[backref=page]{hyperref}

\newcommand{\ovl}{\overline}
\newcommand{\vp}{\varepsilon}

\newcommand{\cp}{M\rtimes_\alpha G}

\newcommand{\Ad}{\mathrm{Ad}\,}
\newcommand{\Aut}{\mathrm{Aut}\,}

\newcommand{\lam}{\ell}
\newcommand{\rartimes}{\rtimes_{\alpha, r}}
\newcommand{\artimes}{\rtimes_\alpha}
\newcommand{\K}{\mathcal{K}}

\newcommand{\Hil}{\mathcal{H}}

\newcommand{\B}{\mathcal{B}}

\newcommand{\Z}{\mathcal{Z}}

\newcommand{\vntensor}{\,\overline{\otimes}\,}

\numberwithin{equation}{section}

\theoremstyle{plain}
\newtheorem{lem}{Lemma}[section]
\newtheorem{pro}[lem]{Proposition}
\newtheorem{thm}[lem]{Theorem}
\newtheorem{cor}[lem]{Corollary}

\theoremstyle{definition}
\newtheorem{defn}[lem]{Definition}

\theoremstyle{remark}
\newtheorem{rem}[lem]{Remark}

\begin{document}

\null\vspace{.5in}
\setcounter{page}{1}
\begin{LARGE}
\begin{center}
A GALOIS CORRESPONDENCE FOR REDUCED  CROSSED PRODUCTS
OF UNITAL SIMPLE C$^*$-ALGEBRAS BY DISCRETE GROUPS
\end{center}
\end{LARGE}
\bigskip

\begin{center}
\begin{tabular}{cc}
{\large Jan Cameron}$^{(*)}$&{ \large Roger R. Smith}$^{(**)}$
\end{tabular}
\end{center}

\vfill

\begin{abstract}
Let a discrete group $G$ act on a unital simple C$^*$-algebra $A$ by outer automorphisms. We establish a Galois correspondence $H\mapsto A\rtimes_{\alpha,r}H$ between subgroups of $G$ and C$^*$-algebras $B$ satisfying $A\subseteq B \subseteq A\rtimes_{\alpha,r}G$, where 
$A\rtimes_{\alpha,r}G$ denotes the reduced crossed product. For a twisted dynamical system $(A,G,\alpha,\sigma)$, we also prove the corresponding result for the reduced twisted crossed product $A\rtimes^\sigma_{\alpha,r}G$.
\end{abstract}

\vfill


\vfill

\noindent 2010 AMS Classification: 46L55, 46L40

\vfill

\noindent ($\ast$) JC was partially supported by Simons Collaboration Grant for Mathematicians \#319001

\noindent ($\ast\ast$) RS was partially supported by Simons Collaboration Grant for Mathematicians \#522375
\newpage

\section{Introduction}\label{intro}
The study of von Neumann algebras and C$^*$-algebras arising from groups and group actions is of central importance in operator algebra theory.  Of particular interest are the crossed product algebras, which date back to the earliest days of the subject when they were introduced by Murray and von Neumann \cite{MvN} to give examples of factors.  There are two parallel strands to the theory, one for von Neumann algebras  and one for C$^*$-algebras.  In both settings, one begins with a noncommutative dynamical system, i.e., an action $\alpha$ of a locally compact group $G$ on a C$^*$-algebra $A$ (respectively, a von Neumann algebra $M$), and constructs a larger C$^*$-algebra $A \rartimes G$ (respectively, a von Neumann algebra $M \artimes G$) by representing both the group and the underlying algebra on a common Hilbert space in such a way that the action is respected by the representation.  In this paper we consider the special case of crossed products arising from the action of a discrete group.

A continuing theme of investigation into discrete crossed products has been to ask to what extent the structure of a group can be recovered from the structure of an associated crossed product algebra.  One instance of this theme is the so-called Galois theory of operator algebras, which relates the structure of subalgebras of a crossed product algebra to the subgroups of its underlying acting group.  When a discrete group $G$ acts by outer automorphisms on a C$^*$-algebra (say, a unital, simple C$^*$-algebra $A$), any subgroup $H$ of $G$ will give rise to a C$^*$-subalgebra $B = A \rtimes_{\alpha, r} H$ of the reduced crossed product $A \rtimes_{\alpha, r} G$ containing $A$.  These are known as \emph{intermediate subalgebras} for the inclusion $A \subseteq A \rartimes G$, and we say that a \emph{Galois correspondence} holds for the inclusion if the map $H \mapsto A \rartimes H$ defines a bijection between subgroups of $G$ and intermediate subalgebras of $A \subseteq A \rartimes G$. 
 
 In the von Neumann algebra setting, work of Choda \cite{Cho78} and Izumi, Longo and Popa \cite{ILP} established a Galois correspondence for inclusions $M \subseteq M \artimes G$ associated to outer actions of discrete groups on separable factors.  
We extended those results to the nonseparable case in \cite{CaSm2}, as a corollary of more general structural results for w$^*$-closed $M$-modules in von Neumann crossed products $M \artimes G$.  In this paper, we examine the analogous problem for C$^*$-algebra crossed products, and our main result, Theorem \ref{intermediate}, establishes a C$^*$-algebraic version of the results in \cite{CaSm2}.  In particular, we obtain a Galois correspondence for inclusions $A \subseteq A \rartimes G$, where $G$ is a discrete group acting by outer automorphisms on a unital simple C$^*$-algebra $A$.   Prior results in this direction include those of Landstad, Olesen and Pedersen \cite{LOP} for abelian groups and  Izumi \cite{Iz} for finite groups.  We also mention the theorem of H. Choda \cite{Cho}, who established the Galois correspondence for discrete groups under the additional hypothesis of the existence of conditional expectations onto intermediate subalgebras of $A \subseteq A \rartimes G$.  Our main result is optimal for discrete crossed products in that it obtains the existence of such conditional expectations as a conclusion, and (as we note in the remarks following Theorem \ref{intermediate}) holds for the largest possible classes of group actions (the outer actions), and underlying C$^*$-algebras (the simple ones).

When $N$ is a normal subgroup of $G$, it is desirable to decompose $A\rtimes_{\alpha,r}G$ as an iterated crossed product by an action of $N$ and then by an action of the quotient group $G/N$ on $A\rtimes_{\alpha,r}N$. However, this is only possible in general when we expand to the larger class of reduced twisted crossed products with the extra ingredient of a cocycle, a map $\sigma$ of $G\times G$ into the unitary group $\mathcal{U}(A)$ of $A$, (see \cite[Theorem 2.1]{Bed}). We show that our results on reduced crossed products extend to include the twisted case.

We now describe the organization of the paper, noting that terminology will be defined in the subsequent sections. A primary goal is to obtain Proposition \ref{average}, which can be viewed as a norm-approximate type of averaging result. Because von Neumann algebras are closed in the $w^*$-topology, such averaging is most effective in this context, and so Section \ref{prelim} is devoted to preliminary lemmas that allow us to recast the necessary averaging arguments for $A$ in terms of those for $A^{**}$. In particular, we make use of two such von Neumann algebra theorems from \cite{ChSi} and \cite{HZ}, and the results of Section \ref{prelim} are designed to allow us to do so.

 In Section \ref{main} we present  our main results on  crossed products of unital simple C$^*$-algebras $A$ by 
outer actions of discrete groups $G$. Proposition \ref{supp} contains a discussion of norm closed $A$-bimodules (which include the intermediate C$^*$-algebras) in preparation for Theorem \ref{intermediate}, where we show that the intermediate C$^*$-algebras are all of the form $A\rtimes_{\alpha,r}H$ for subgroups $H\subseteq G$. We also include some consequences; for example, we recapture a theorem of Kishimoto \cite{Ki} that shows simplicity of $A\rtimes_{\alpha,r}G$ when $A$ is unital and simple and $G$ acts by outer automorphisms. In Section \ref{tcp}, we extend the results of Section \ref{main} to reduced twisted crossed products. The ideas are essentially those of Section \ref{main}, but the following  extra ingredients are now required. We  transfer a result of Packer and Raeburn \cite[Lemma 3.3]{PR} on exterior equivalence for full twisted C$^*$-algebra crossed products to the setting of twisted von Neumann algebra crossed products
(Lemma \ref{exterior}) and we include a more general version of a theorem of Sutherland \cite[Theorem 5.1]{Sut} which untwists a twisted crossed product of a von Neumann algebra by tensoring by a copy of $\B(\ell^2(G))$ (Lemma \ref{untwist}). These allow us 
to extend our methods for the untwisted case in Section \ref{main}, and to establish our Galois result for reduced twisted crossed products in Theorem \ref{maintw}.

\section{Preliminary lemmas}\label{prelim}

In this section we present a sequence of preliminary results to be used in proving Proposition \ref{average} which in turn leads to our first main result in Theorem \ref{intermediate}. In general terms, they can
be described as investigating various types of averaging in the central summands of $A^{**}$. Ultimately, we wish to apply averaging within a C$^*$-algebra $A$ (Proposition \ref{average}) but our
methods require us to do this first in $A^{**}$. Lemma \ref{WtoN} will allow us to pass back and forth between these two situations.
An important technical tool for our work is to show that, in appropriate circumstances, 0 is in the closure of the sets introduced in Definition \ref{defW}. The germ of this idea for C$^*$-algebras originates in the work of Elliott \cite[Theorem 2.3]{Ell}.

\begin{defn}\label{defW}

Let $\alpha$ be a $*$-automorphism of a unital C$^*$-algebra $A$. For $x\in A$, we define a convex set by

\begin{equation}\label{W}
W_A(x,\alpha)=\left\{\sum_{i=1}^n a_i^*x\alpha(a_i):n\geq 1,\ a_i\in A,\ \sum_{i=1}^n a_i^*a_i=1\right\}\subseteq A.
\end{equation}

\end{defn}

\begin{lem}\label{WtoN}
Let $\alpha$ be a $*$-automorphism of a unital C$^*$-algebra $A$, and let $x\in A$. 
\begin{itemize}
\item[(i)] If $A$ is represented nondegenerately on a Hilbert space $\Hil$ and $\alpha$ can be extended to a $*$-automorphism $\tilde\alpha$ of $A''$, then
$\overline{W_A(x,\alpha)}^{w^*}=\overline{W_{A''}(x,\tilde\alpha)}^{w^*}$.
\item[(ii)] If $0\in \overline{W_{A^{**}}(x,\alpha^{**})}^{w^*}$, then $0\in \overline{W_A(x,\alpha)}^{\|\cdot\|}$.
\end{itemize} 
\end{lem}

\begin{proof}
(i). Clearly $\overline{W_A(x,\alpha)}^{w^*}\subseteq \overline{W_{A''}(x,\tilde\alpha)}^{w^*}$. To prove the reverse containment consider an element $y_0\notin \overline{W_A(x,\alpha)}^{w^*}$; we will show that $y_0\notin \overline{W_{A''}(x,\tilde\alpha)}^{w^*}$.
Since $\overline{W_A(x,\alpha)}^{w^*}$ is convex and $w^*$-compact, there exist vectors $\xi_1,\ldots,\xi_n,\eta_1,\ldots,\eta_n\in \Hil$ and $\vp >0$ so that the linear functional $\phi(\cdot)=\sum_{i=1}^n\langle \cdot \xi_i,\eta_i\rangle $ on $A''$ satisfies $\mathrm{Re}\,\phi(y-y_0)\geq \vp$ for $y\in \overline{W_A(x,\alpha)}^{w^*}$, by Hahn-Banach separation. If $y_1\in W_{A''}(x,\tilde\alpha)$, then there exist $m_1,\ldots,m_k\in A''$ so that $\sum_{i=1}^km_i^*m_i=1$ and $y_1=\sum_{i=1}^km_i^*x\tilde\alpha(m_i)$.
If we view these operators $m_i\in A''$  as the entries of a column matrix $C$ in $A''
\,\ovl{\otimes}\,\mathcal{B}(\ell^2(\mathbb{N}))$, then the Kaplansky density theorem gives a net $\{C_\lambda\}_{\lambda\in \Lambda}$ of finitely nonzero columns with entries from $A$
with $\|C_\lambda\|\leq 1$ and
$\lim_{\lambda\in\Lambda}C_\lambda=C$ strongly. Since $C^*C=1$, we see that $\|C\zeta\|=\|\zeta\|$ for all $\zeta\in\Hil$, and so $\lim_{\lambda}\|C_\lambda \zeta\|=\|\zeta\|$. From this
it follows that $\lim_{\lambda}
\|(1-C_{\lambda}^*C_\lambda)^{1/2}\zeta\|=0$. If we modify $C_\lambda$ by replacing a 0 entry by $(1-C_{\lambda}^*C_\lambda)^{1/2}$, then we may further assume that
$C_{\lambda}^*C_\lambda=1$. Since $*$-automorphisms of von Neumann algebras are strongly continuous on norm bounded sets, it follows that
$\lim_{\lambda}\alpha\otimes\mathrm{id}(C_\lambda)=
\tilde\alpha\otimes\mathrm{id}(C)$ strongly, and we see that $y_1$ is the strong limit of a net from $W_A(x,\alpha)$. Thus $\mathrm{Re} \,\phi(y_1-y_0)\geq \vp$, and so $y_0\notin \overline{W_{A''}(x,\tilde\alpha)}^{w^*}$. This proves (i).

(ii). Assume that $A$ is in its universal representation on a Hilbert space $\Hil$, so that $A^{**}=A''$. To arrive at a contradiction, suppose that $ 0\in \overline{W_{A^{**}}(x,\alpha^{**})}^{w^*}$ but that $0\notin \overline{W_A(x,\alpha)}^{\|\cdot\|}$. By the
separation form of the Hahn-Banach theorem, there exists $\phi\in A^*$ so that
\begin{equation}\label{WtoN.1}
\mathrm{Re}\,\phi(y)\geq 1,\ \ \ y\in W_A(x,\alpha).
\end{equation}
In this representation, each bounded linear functional on $A$ is a vector functional, so there exist $\xi,\eta\in \Hil$ so that
\begin{equation}\label{WtoN.2}
\mathrm{Re}\,\langle y\xi,\eta\rangle=\mathrm{Re}\,\phi(y)\geq 1,\ \ \ y\in W_A(x,\alpha).
\end{equation}
 Then this inequality also holds for $y\in \overline{W_A(x,\alpha)}^{w^*}$ by $w^*$-continuity of the vector functional $\langle \cdot\,\xi,\eta\rangle$, and part (i) shows that $
0\notin \overline{W_{A^{**}}(x,\alpha^{**})}^{w^*}$, a contradiction that proves (ii).
\end{proof}

The next two results are designed for use in Lemma \ref{no2-1}, and prepare the ground for applying a theorem of Haagerup and Zsid\'{o} \cite{HZ} (see \cite{Hal,Str,SZ} for related work).

\begin{lem}\label{Zmap}
Let $A$ be a unital simple C$^*$-algebra represented nondegenerately on a Hilbert space $\Hil$, let $M=A''$, let $J\subseteq M$ be a norm closed two sided ideal satisfying $J\cap
\Z(M)=\{0\}$, and let $u\in M$ be a unitary that implements a $*$-automorphism $\alpha$ of $A$. Suppose that there exist elements $z\in \Z(M)$, $j\in J$, and $a\in A$ so that
\begin{equation}\label{Zmap.1}
\|z+j+au\|<\|z\|.
\end{equation}
Then $\alpha$ is inner.
\end{lem}

\begin{proof}
We examine several cases of \eqref{Zmap.1}.

\medskip

\noindent Case 1:\quad $z=1$ and $a\geq 0$.

\medskip

It follows from \eqref{Zmap.1} that
\begin{equation}\label{Zmap.2}
\|u^*+ju^*+a\|<1.
\end{equation}
Let $\pi:M\to M/J$ be the quotient map, and note that $\pi|_A$ is an isomorphism since $A$ is simple. Then \eqref{Zmap.2} gives
\begin{equation}\label{Zmap.3}
\|\pi(u^*)+\pi(a)\|<1
\end{equation}
in $M/J$. Fix $\lambda$ in the spectrum $\sigma(\pi(u^*))$ and choose a state $\phi$ on $M/J$ such that $\phi(\pi(u^*))=\lambda$. Then, from \eqref{Zmap.3},
\begin{equation}\label{Zmap.4}
|\lambda+\phi(\pi(a))|<1.
\end{equation}
Since $\phi(\pi(a))\geq 0$, we conclude that $\mathrm{Re}\, \lambda <0$ otherwise \eqref{Zmap.4} is contradicted. Thus
\begin{equation}\label{Zmap.5}
\sigma(-\pi(u^*))\subseteq\{e^{i\theta}:-\pi/2<\theta<\pi/2\},
\end{equation}
so we can define $b=\log(-\pi(u^*))\in M/J$ using the standard analytic branch of the logarithm. Let $\delta$ be the inner derivation of $M/J$ defined by
\begin{equation}\label{Zmap.6}
\delta(\pi(m))=\pi(m)b-b\pi(m),\ \ \ m\in M.
\end{equation}
By \cite[Lemmas 18.8, 18.14]{BD}, $e^\delta$ is the automorphism $\Ad(-\pi(u))=\Ad(\pi(u))\in\Aut(M/J)$. Moreover, $\delta$ is the norm limit of polynomials in $\Ad(\pi(u))$ and consequently
maps $\pi(A)$ to itself. By \cite{Sak}, $\delta|_{\pi(A)}$ is implemented by an element $\pi(t)\in \pi(A)$ for some $t\in A$. Since $\alpha$ maps $J$ to $J$, it induces an automorphism $\tilde{\alpha}$ of
$M/J$, and we have $\tilde{\alpha}=\Ad(e^{\pi(t)})$. If $v$ is the unitary in the polar decomposition in $M$ of $e^t$, then $v\in A$ since $e^t$ is invertible in $A$ and $\pi(v)$
implements $\tilde{\alpha}$. Thus, for $x\in A$,
\begin{equation}\label{Zmap.7}
vxv^*-uxu^*\in J,
\end{equation}
and so $\alpha=\Ad v$ since $A\cap J=\{0\}$ by simplicity of $A$. Thus $\alpha$ is an inner automorphism of $A$.

\medskip

\noindent Case 2:\quad $z=1$ and no restriction on $a$.

\medskip

\noindent Since
\begin{equation}\label{Zmap.8}
\|1+j+au\|<1
\end{equation}
from \eqref{Zmap.1}, $\pi(au)$ is invertible in $M/J$ so $\pi(a)$ is invertible in $M/J$. Thus $\pi(a)$ is invertible in $\pi(A)$ so simplicity of $A$ shows  that $a$ is invertible in $A$.
The polar decomposition $a=(aa^*)^{1/2}v$ then occurs in $A$ so $v\in A$. From \eqref{Zmap.8},
\begin{equation}\label{Zmap.9}
\|1+j+(aa^*)^{1/2}vu\|<1,
\end{equation}
and $vu$ implements an automorphism of $A$. By Case 1, with $(aa^*)^{1/2}$ and $vu$ replacing respectively $a$ and $u$, we see that $\Ad vu$ is inner on $A$ and so $\alpha=\Ad u$ is also
inner on $A$.

\medskip

\noindent Case 3:\quad $z\geq 0$ and no restriction on $a$.

\medskip

\noindent From \eqref{Zmap.1}, it is clear that $z\ne 0$, so by scaling we may assume that $\|z\|=1$. Then \eqref{Zmap.1} becomes
\begin{equation}\label{Zmap.10}
\|z+j+au\|<1.
\end{equation}
Choose $\vp>0$ to be so small that
\begin{equation}\label{Zmap.11}
\|z+j+au\|+\vp<1,
\end{equation}
 let $p_\vp\in \Z(M)$ be the spectral projection of $z$ for the interval $[1-\vp,1]$ and note that $p_\vp \notin J$ since $J\cap\Z(M)=\{0\}$ by hypothesis. Then $Jp_\vp \cap \Z(Mp_\vp)=\{0\}$ and
\begin{align}
\|p_\vp+j p_\vp  +aup_\vp\|&=\|(z+j+au)p_\vp +(1-z)p_\vp\|\notag\\
&\leq \|z+j+au\|+\vp<1.\label{Zmap.12}
\end{align}

Simplicity of $A$ implies that $x\mapsto xp_\vp$ represents $A$ faithfully on $p_\vp \Hil$ and $up_\vp$ is a unitary in $Mp_\vp$ implementing an automorphism of $Ap_\vp$. By Case 2 there
exists a unitary $v\in A$ so that $\Ad vp_\vp=\Ad up_\vp$ on $Ap_\vp$. By faithfulness of the representation, $\Ad v=\Ad u$ on $A$ and so $\alpha$ is inner.

\medskip

\noindent Case 4:\quad the general case.

\medskip

\noindent Since $\Z(M)$ is an abelian von Neumann algebra, there exists a unitary $v\in \Z(M)$ so that $z=|z|v$. Then
\begin{equation}\label{Zmap.13}
\|\,|z|+jv^*+auv^*\|<\|z\|=\|\,|z|\,\|,
\end{equation}
and we are now in Case 3 with $|z|$, $jv^*$ and $uv^*$ replacing respectively $z$, $j$ and $u$. Since $uv^*$ also implements $\alpha$, we conclude that $\alpha$ is inner, completing the
proof.
\end{proof}

\begin{cor}\label{Zmap2}
Let $A$ be a unital simple C$^*$-algebra represented nondegenerately on a Hilbert space $\Hil$, let $M=A''$, let $J\subseteq M$ be a norm closed two sided ideal satisfying
$J\cap\Z(M)=\{0\}$, and let $u\in M$ be a unitary that implements an outer $*$-automorphism of $A$. Then there exists a unital completely positive $\Z(M)$-bimodular map $\Phi:M\to \Z(M)$
such that
\begin{equation}\label{Zmap2.1}
\Phi|_{(J+Au)}=0.
\end{equation}
\end{cor}

\begin{proof}
Since $\Ad u|_A$ is not inner, Lemma \ref{Zmap} shows that
\begin{equation}\label{Zmap2.2}
\|z+j+au\|\geq \|z\|,\ \ \ z\in \Z(M),\ \ j\in J,\ \ a\in A.
\end{equation}
Thus there is a unital contraction $\Psi:\Z(M)+J+Au\to \Z(M)$ given by
\begin{equation}\label{Zmap2.3}
\Psi(z+j+au)=z,\ \ \ z\in \Z(M),\ \ j\in J,\ \ a\in A.
\end{equation}
Since $\Z(M)$ is abelian, norms in $\mathbb{M}_n(\Z(M))$ are calculated by applying characters $\omega$ to the entries of matrices. Each $\omega\circ\Psi$ is a contractive unital linear functional on $\Z(M)+J+Au$ so extends to a state on $M$, and thus is completely contractive. From this we see that $\Psi$ is a complete contraction so, by injectivity of $\Z(M)$, it extends to a completely contractive unital map $\Phi:M\to \Z(M)$ which is hence completely positive. Since $\Phi|_{\Z(M)}=\mathrm{id}$, $\Phi$ is a conditional expectation onto $\Z(M)$ and so is $\Z(M)$-bimodular. By construction, $\Phi|_{(J+Au)}=0$, completing the proof.
\end{proof}

Recall that an automorphism $\beta$ of a von Neumann algebra $M$ is {\it{properly outer}} if there does not exist a nonzero $\beta$-invariant central projection so that $\beta|_{Mz}$ is inner.
If $\alpha$ is an  automorphism of a C$^*$-algebra $A$, then it lifts to an automorphism $\alpha^{**}$ of $A^{**}$. The next lemma addresses the central summand on which $\alpha^{**}$ is
properly outer, while Lemmas \ref{no2-1} and \ref{lem2.12} apply respectively to the purely infinite and finite central summands on which $\alpha^{**}$ is inner. The latter requires the special cases of a
finite factor in Lemma \ref{2-1factor} and of a separably acting II$_1$ von Neumann algebra in Lemma \ref{lem2.10}. Phillips \cite{Ph} has given an example of an outer
automorphism of C$^*_r(\mathbb{F}_2)$ that is implemented by a unitary in $L(\mathbb{F}_2)$, so all three of these central summands just discussed may well be present.

We note that we are using the term {\it outer} automorphism to mean one that is not inner, while some papers have slightly different formulations. For example, Kishimoto \cite{Ki} expresses his results in terms of the condition that the strong Connes spectrum (from \cite{Ki0}) should not be $\{1\}$. In the case of a unital simple C$^*$-algebra, this is equivalent to outerness in our sense using results from \cite{Ki0,Ole}.

\begin{lem}\label{propouter}
Let $A$ be a unital  C$^*$-algebra represented nondegenerately on a Hilbert space $\Hil$ and let $M=A''$. If $\alpha$ is a $*$-automorphism of $A$ that extends to a properly outer
$*$-automorphism $\tilde\alpha$ of $M$, then $0\in \overline{W_A(x,\alpha)}^{w^*}$ for all $x\in A$.
\end{lem}

\begin{proof}
Fix $x_0\in A$ and consider the completely bounded map $\phi:M\to M$ defined by
\begin{equation}\label{propouter.1}
\phi(x)=\tilde\alpha^{-1}(x)\tilde\alpha^{-1}(x_0),\ \ \ x\in M.
\end{equation}
Since $\tilde\alpha^{-1}$ is properly outer, \cite[Lemma 4.2]{CaSm} and the discussion of \cite[Theorem 3.3]{ChSi} that precedes it in \cite[Section 4]{CaSm} give a net indexed by sets of
operators $\beta =(m_j)_{j\in J}$ from $M$ satisfying $\sum_{j\in J}m_j^*m_j=1$ so that
\begin{equation}\label{propouter.2}
w^*{\text{-}}\lim_\beta \sum_{j\in J}\phi(xm_j^*)m_j=0,\ \ \ x\in M.
\end{equation}
Applying the $w^*$-continuous map $\tilde\alpha$ to \eqref{propouter.2} leads to
\begin{equation}\label{propouter.3}
w^*{\text{-}}\lim_\beta \sum_{j\in J} xm_j^*x_0\tilde\alpha(m_j)=0,
\ \ \ x\in M,
\end{equation}
so we may take $x=1$ in \eqref{propouter.3} to obtain
\begin{equation}\label{propouter.4}
w^*{\text{-}}\lim_\beta \sum_{j\in J} m_j^*x_0\tilde\alpha(m_j)=0.
\end{equation}
The Kaplansky density argument of Lemma \ref{WtoN} allows us to assume that the sums in \eqref{propouter.4} have only finitely many terms, and so $0\in \ovl{W_{A''}(x_0,\tilde\alpha)}^{w^*}$. Then $0\in
\ovl{W_A(x_0,\alpha)}^{w^*}$ by Lemma \ref{WtoN} (i).
\end{proof}

\begin{lem}\label{no2-1}
Let $A$ be a unital simple C$^*$-algebra represented nondegenerately on a Hilbert space $\Hil$ and let $M=A''$. Let $\alpha$ be an outer $*$-automorphism of $A$ that is implemented by a
unitary $u\in M$. If $M$ has no II$_1$ part, then $0\in \overline{W_A(x,\alpha)}^{w^*}$ for each $x\in A$.
\end{lem}

\begin{proof}
Let $J\subseteq M$ be the norm closed two sided ideal generated by the finite projections of $M$. If $J\cap \Z(M)\ne \{0\}$, then $\Z(M)$ contains a nonzero finite projection $p$ and $Mp$ is a finite von Neumann algebra. By hypothesis, $M$ has no II$_1$ part and so $Mp$ is type I and finite. Simplicity of $A$ shows that the representation $a \mapsto ap$ of $A$ is faithful and so $A$ must be finite dimensional, and thus a matrix factor. This is impossible because all automorphisms would then be inner, so we conclude that $J\cap \Z(M)=\{0\}$. In addition, this argument shows that $M$ also has no type I$_{\mathrm{finite}}$ part.

Following the notation of \cite{HZ}, $S_{\Z(M)}(M)$ is the set of unital positive $\Z(M)$-bimodule maps $\Phi:M\to \Z(M)$, and the
  {\it{central image}} of $x\in M$ modulo $J$ is 
\begin{equation}\label{no2-1.1}
V_{M,J}(x):=\{\Phi(x):\Phi\in S_{\Z(M)}(M),\ \Phi|_J=0\}.
\end{equation}
Since we have already established that $M$ has no finite part, \cite[Th\'{e}or\`{e}me 3]{HZ} reduces to 
\begin{equation}\label{no2-1.2}
C_M(x)=V_{M,J}(x)
\end{equation}
where, for $x\in M$, 
\begin{equation}\label{no2-1.3}
C_M(x):=\ovl{{\mathrm{conv}}}^{\|\cdot\|}\{v^*xv:v\in M,\ {\mathrm{unitary}}\}\cap \Z(M).
\end{equation}

If $x\in A$, then Corollary \ref{Zmap2} gives a map $\Phi\in S_{\Z(M)}(M)$ so that $\Phi|_J=0$ and $\Phi(xu)=0$. Thus 
$0\in V_{M,J}(xu)$ so $0\in C_M(xu)$ from \eqref{no2-1.2}. Given $\delta >0$, we see from \eqref{no2-1.3} that there are positive constants $\lambda_1,\ldots ,\lambda_n \in \mathbb{R}$ summing to 1 and unitaries $v_1,\ldots,v_n\in M$ so that
\begin{equation}\label{no2-1.4}
\left\|\sum_{i=1}^n \lambda_i v_i^* xu v_i\right\| <\delta.
\end{equation}
Then, from \eqref{no2-1.4},
\begin{equation}\label{no2-1.5}
\left\| \sum_{i=1}^n (\sqrt{\lambda_i}v_i)^*x\tilde{\alpha}(\sqrt{\lambda_i}v_i)\right\|=
\left\|\sum_{i=1}^n \lambda_i v_i^* xu v_iu^*\right\| <\delta,
\end{equation}
where $\tilde{\alpha}$ denotes the extension of $\Ad u$ from $A$ to $M$. Since $\delta >0$ was arbitrary, we see that $0\in \ovl{W_M(x,\tilde\alpha)}^{\|\cdot\|}\subseteq \ovl{W_M(x,\tilde\alpha)}^{w^*}$, so $0\in \ovl{W_A(x,\alpha)}^{w^*}$ by Lemma \ref{WtoN} (i).
\end{proof}

We recall that if $A$ is a Banach algebra with unit 1, then the {\em{state space}} is $\{\phi\in A^*:\|\phi\|=\phi(1)=1\}$. The {\em{numerical range}} of $a\in A$ is the set of values of all states applied to $a$ and is a compact convex subset of $\mathbb{C}$.

\begin{lem}\label{2-1factor}
Let $M$ be a II$_1$ factor with normalized trace $\tau$ and let $x\in M$ be an element whose numerical range contains 0. Given $\delta >0$, there exist elements $m_1,\ldots ,m_k\in M$ satisfying $\sum_{i=1}^k
m_i^*m_i=1$ and
\begin{equation}\label{2-1factor.1}
\left\|\sum_{i=1}^k m_i^*xm_i\right\| < \delta.
\end{equation}
\end{lem}

\begin{proof}
Fix $\delta >0$, and let $x\in M$ be an element whose numerical range contains 0. As in the proof of \cite[Proposition 3.1]{Cam}, we may assume that $M$ has a separable predual. There exists a state $\psi\in M^*$ so that $\psi(x)=0$, and a Hahn-Banach separation argument shows that the normal states on $M$ are $w^*$-dense in the state space. Consequently there exists a normal state $\phi$ so that $|\phi(x)|<\delta$. Then there exists $b\in L^1(M)^+$ so that $\tau(b)=1$ and $\phi(\cdot)=\tau(b\,\cdot)$. If we replace $b$ by $(b\wedge r1)/\tau(b\wedge r1)$ for a sufficiently large value of $r$, then we may assume that $b\in M^+$. Now choose a masa in $M$ containing $b$, and identify this masa with $L^\infty([0,1],dt)$ by \cite[Theorem 3.5.2]{SS2} which requires separability of $M_*$. A further approximation allows us to choose from the masa $n$ orthogonal  projections summing to 1 that are equivalent in $M$ and to assume that $b$ is in their span. Let $\mathbb{M}_n$ be a matrix subfactor so that these projections are the minimal diagonal projections.

Let $E:M\to \mathbb{M}_n$ be the trace preserving conditional expectation. Then
\begin{equation}\label{2-1factor.2}
\tau(bE(x))=\tau(E(bx))=\tau(bx),
\end{equation}
so $|\tau(bE(x))|<\delta$. On $\mathbb{M}_n$, the map $y\mapsto \tau(by)I_n$ is a unital completely positive map, and so there are matrices $c_1,\ldots ,c_s\in \mathbb{M}_n$ so that 
\begin{equation}\label{2-1factor.3}
\tau(by)I_n=\sum_{i=1}^sc_i^*yc_i,\ \ \ y\in \mathbb{M}_n,
\end{equation}
and $\sum_{i=1}^sc_i^*c_i=I_n$, \cite[Theorem 1]{Choi}. Then
\begin{align}
\left|\tau\left(\sum_{i=1}^sc_i^*xc_i\right)\right|&=\left|\tau E
\left(\sum_{i=1}^sc_i^*xc_i\right)\right| 
= \left|\tau\left(\sum_{i=1}^sc_i^*E(x)c_i\right)\right| \notag\\
&=|\tau(bE(x))|<\delta.\label{2-1factor.4}
\end{align}
By the Dixmier approximation theorem \cite[Theorem 8.3.6]{KR2}, there are positive constants $\{\lambda_j\}_{j=1}^m$ summing to 1 and unitaries $u_j\in M$ so that 
\begin{equation}\label{2-1factor.5}
\left\|\sum_{j=1}^m \lambda_j u_j^*\left(\sum_{i=1}^sc_i^*xc_i\right)u_j\right\| < \delta,
\end{equation}
 and the result follows by listing the set of elements $\{\sqrt{\lambda_j}c_iu_j:1\leq j\leq m,\,1\leq i\leq s\}$ as $m_1,\ldots ,m_k\in M$.
\end{proof}

Our next objective is to establish an analogous result for general II$_1$ von Neumann algebras in Lemma \ref{lem2.12}. This requires several preliminary lemmas.

Let $A$ be a unital C$^*$-algebra. The unitary group of $\mathbb{M}_n(A)$ is denoted $\mathcal{U}_n(A)$ while $\mathcal{U}_n^{{\mathrm{her}},\pi}(A)$ denotes the set of unitaries in
$\mathcal{U}_n(A)$ of the form $e^{iH}$where $H$ is hermitian in $\mathbb{M}_n(A)$ and $\|H\|\leq \pi$. It is convenient to work with
$\mathcal{U}_n^{{\mathrm{her}},\pi}(A)$ rather than  $\mathcal{U}_n(A)$ because unitaries of the form $e^{iH}$ in quotient algebras $A/J$ can be lifted to $A$ by lifting the associated hemitian elements whereas general unitaries do not always lift.

The space of column vectors $(a_1,\ldots ,a_n)^t$ over $A$ satisfying
$\sum_{i=1}^n a_i^*a_i=1$ is denoted ${\mathrm{Col}}_{n,1}(A)$.

\begin{lem}\label{lem2.8}
Let $B\subseteq \B(\Hil)$ be a separable unital C$^*$-algebra whose $w^*$-closure $M$ is a II$_1$ von Neumann algebra. Let $n \geq 1$, and let $\mathcal{S}$ be a countable norm dense subset
of $\mathcal{U}_n^{{\mathrm{her}},\pi}(B)$. Let $\mathcal{C}\subseteq {\mathrm{Col}}_{n,1}(B)$ be the countable set of first columns of matrices in $\mathcal{S}$. Then $\mathcal{C}$ is strongly dense
in ${\mathrm{Col}}_{n,1}(M)$.
\end{lem}

\begin{proof}
Let $C\in {\mathrm{Col}}_{n,1}(M)$, and regard $C$ as the first column of a matrix in $\mathbb{M}_n(M)$ whose other columns are all 0. Then $C^*C$ is the projection
$P={\mathrm{diag}}(1,0,\ldots,0)\in \mathbb{M}_n(M)$, so $CC^*$ is a projection $Q\in \mathbb{M}_n(M)$ which is equivalent to $P$. Since $\mathbb{M}_m(M)$ is finite, $I-P$ and $I-Q$ are
equivalent by a partial isometry $E$ such that $E^*E=I-P$ and $EE^*=I-Q$. Then $C+E$ is a unitary $U\in \mathbb{M}_n(M)$ whose first column is $C$. By the functional calculus, we may
choose a self-adjoint element $H\in \mathbb{M}_n(M)$ with $\|H\|\leq \pi$ so that $U=e^{iH}$. By the Kaplansky density theorem, there is a net $\{H_\lambda\}_{\lambda\in \Lambda}$ of
self-adjoint elements in $\mathbb{M}_n(B)$ such that $\lim_{\lambda}H_\lambda=H$ strongly and $\|H_\lambda\|\leq \pi$. The functional calculus is strongly continuous on uniformly bounded
convergent nets, so $\lim_\lambda e^{iH_\lambda}=e^{iH}=U$ strongly. By norm density of $\mathcal{S}$ in $\mathcal{U}_n^{{\mathrm{her}},\pi}(B)$, the net $\{e^{iH_\lambda}\}_{\lambda\in \Lambda}$ can
be replaced by a net from $\mathcal{S}$. The result now follows by dropping to the first columns of these matrices.
\end{proof}

\begin{lem}\label{lem2.9}
Let $A$ be a separable unital C$^*$-algebra, let $n\geq 1$, let $\mathcal{S}_n$ be a countable norm dense subset of $\mathcal{U}_n^{{\mathrm{her}},\pi}(A)$, and let $\mathcal{C}_n$ be the countable
set of first columns of matrices in $\mathcal{S}_n$. If $\sigma :A\to \B(\Hil_\sigma)$ is a $*$-representation such that $\sigma(A)''$ is type II$_1$, then $\sigma\otimes
I_n(\mathcal{C}_n)$ is strongly dense in ${\mathrm{Col}}_{n,1}(\sigma(A)'')$.
\end{lem}

\begin{proof}
Define $B$ to be $\sigma(A)$, isomorphic to $A/\ker \sigma$. Each self-adjoint $H\in \mathbb{M}_n(B)$ with $\|H\|\leq \pi$ lifts to a self-adjoint $\tilde{H}\in \mathbb{M}_n(A)$ with
$\|\tilde{H}\|\leq \pi$ and thus $\sigma\otimes I_n$ maps $\mathcal{U}_n^{{\mathrm{her}},\pi}(A)$ onto $\mathcal{U}_n^{{\mathrm{her}},\pi}(B)$. Then $\sigma\otimes I_n(\mathcal{S}_n)$ is norm dense in
$\mathcal{U}_n^{{\mathrm{her}},\pi}(B)$, and the result now follows from Lemma \ref{lem2.8}.
\end{proof}

In the next lemma we will make use of direct integral theory for von Neumann algebras on separable Hilbert spaces, and we refer to \cite[Ch. 14]{KR2} for a discussion of this topic. We
briefly review the results that we will need. The Hilbert space $\Hil$ is decomposed as $\int^\oplus_\Omega \Hil(\omega)\,d\mu(\omega)$ of Hilbert spaces indexed by a probability space
$(\Omega,\mu)$, and the von Neumann algebra $M$ decomposes as a direct integral of factors $\int^\oplus_\Omega M(\omega)\,d\mu(\omega)$. We will only be concerned with a II$_1$ von
Neumann algebra so, after deleting a null set, each $M(\omega)$ is a II$_1$ factor. If $A$ is a separable unital C$^*$-subalgebra of $M$ and the components of $a\in A$ are written
$a(\omega)\in M(\omega)$, then let $A(\omega)\subseteq M(\omega)$ be the C$^*$-algebra generated by $\{a(\omega):a\in A\}$. After deletion of a countable number of null sets, the maps
$a\mapsto a(\omega)$ define surjective $*$-homomorphisms $\sigma_\omega:A\to A(\omega)$ for $\omega\in\Omega$ \cite[Theorem 14.1.13]{KR2}. Moreover, if $A$ is strongly dense in $M$, then each $A(\omega)$ is
strongly dense in $M(\omega)$ after a further deletion of a null set.

For a unital C$^*$-algebra $A$, let ${\mathrm{Col}}_{\infty,1}(A)$ denote the set of finitely nonzero columns over $A$ satisfying $\sum_{i=1}^\infty a_i^*a_i=1$.

\begin{lem}\label{lem2.10}
Let $M$ be a II$_1$ von Neumann algebra acting on a separable Hilbert space $\Hil$. Let $x\in M$ and suppose that there is a unital completely positive $\mathcal{Z}(M)$-bimodule map $\Phi
: M\to \mathcal{Z}(M)$ such that $\Phi(x)=0$. Let $W\subseteq M$ be the convex set
\begin{equation}
W=\{C^*XC: C\in {\mathrm{Col}}_{\infty,1}(M)\},
\end{equation}
where $X={\mathrm{diag}}(x,x,x,\ldots)$. Then $0\in \overline{W}^{w^*}$.
\end{lem}

\begin{proof}
We argue by contradiction, so suppose that $0\notin \overline{W}^{w^*}$. Since $W$ is a uniformly bounded convex set, the weak operator and $w^*$-topologies coincide on $W$ so, by
Hahn-Banach separation, there exist vectors $\xi_1,\ldots,\xi_r,\eta_1,\ldots,\eta_r\in \Hil$ so that the linear functional $\psi\in M_*$ defined by
$\psi(\cdot)=\sum_{i=1}^r\langle\cdot\xi_i,\eta_i\rangle$ satisfies
\begin{equation}\label{eq10.2}
{\mathrm{Re}}\,\psi(w)\geq 1,\ \ \ w\in W.
\end{equation}
There is a probability space $(\Omega,\mu)$ so that $M$ has a direct integral factor decomposition
\[ M=\int^\oplus_\Omega M(\omega)\,d\mu(\omega)\ \ \ {\mathrm{on}}\ \ \ \Hil =
\int^\oplus_\Omega \Hil(\omega)\,d\mu(\omega).\]
If $\Omega_k=\{\omega: \|\xi_i(\omega)\|,\|\eta_i(\omega)\|\leq k,\ 1\leq i\leq r\}$, then $\lim_{k\to\infty}\mu(\Omega_k)=1$, so a suitably large choice of $k$ allows us to multiply
each $\xi_i,\eta_i$ by $\chi_{\Omega_k}$ and further assume that there is a constant $K$ so that $\|\xi_i(\omega)\|,\|\eta_i(\omega)\|\leq K$ for $\omega\in \Omega$ and $1\leq i\leq r$.

Now choose a unital separable C$^*$-algebra $A\subseteq M$ which is strongly dense in $M$ and also contains the element $x$. Then $A(\omega)$ is stongly dense in $M(\omega)$ for almost
all $\omega$, so we may discard a null set and assume that this is always the case.

Following the proof of \cite[Theorem 4.4]{ChatS}, there are states $\phi(\omega)$ on $M(\omega)$ so that
\begin{equation}\label{10.3}
\Phi(x)(\omega)=\phi(\omega)(x(\omega))
\end{equation}
for almost all $\omega\in\Omega$, and we may assume that this is always true by discarding a null set. Then 0 is in the numerical range of $x(\omega)$ for each $\omega\in \Omega$ since
$\Phi(x)=0$. Now define $\delta=(3rK^2)^{-1}$ and let $\omega_0\in \Omega$ be fixed but arbitrary. By Lemma \ref{2-1factor}, there exists $C(\omega_0)\in {\mathrm{Col}}_{n,1}(M(\omega_0))$ so
that
\begin{equation}\label{10.4}
\|C(\omega_0)^*X(\omega_0)C(\omega_0)\|<\delta.
\end{equation}
In particular,
\begin{equation}\label{10.5}
\left|\sum_{i=1}^r\langle C(\omega_0)^*X(\omega_0)C(\omega_0)\xi_i(\omega_0),\eta_i(\omega_0)\rangle
\right|<\delta rK^2.
\end{equation}
Now let $\mathcal{C}_n$ be the countable set of columns of length $n$ over $A$ constructed in Lemma \ref{lem2.8} and list the elements of this set as $\{C_{n,j}: j\geq 1\}$. After deleting a
null set, we may assume that there are surjective $*$-homomorphisms $\sigma_\omega : A\to A(\omega)$ defined by $a\mapsto a(\omega)$. From the conclusion of strong density in Lemma
\ref{lem2.8}, we may replace $C(\omega_0)$ in \eqref{10.5} by $\sigma_{\omega_0}\otimes I_n(C_{n,j}(\omega_0))\in {\mathrm{Col}}_{n,1}(A(\omega_0))$ for a suitable choice of $j$, while
preserving the inequality. Since $\omega_0\in \Omega$ was arbitrary, we conclude that the inequality
\begin{equation}\label{10.6}
\left|\sum_{i=1}^r\langle C_{n,j}(\omega_0)^*X(\omega_0)
C_{n,j}(\omega_0)\xi_i(\omega_0),\eta_i(\omega_0)\rangle \right|<\delta rK^2
\end{equation}
is valid for at least one pair of integers $(n,j)$ depending on $\omega_0$. Define, for $n,j\geq 1$,
\begin{equation}\label{10.7}
E_{n,j}=\left\{\omega\in\Omega: \left|\sum_{i=1}^r \langle
C_{n,j}(\omega)^*X(\omega)C_{n,j}(\omega)\xi_i(\omega),\eta_i(\omega)
\rangle \right|<\delta rK^2\right\}.
\end{equation}
Each is a measurable set and, from \eqref{10.6},
$\Omega=\cup\{E_{n,j}: n,j\geq 1\}$. Now disjointify the $E_{n,j}$'s and write $\Omega$ as a countable disjoint union of measurable sets $F_{n,j}\subseteq E_{n,j}$, $n,j\geq 1$. These correspond
to pairwise orthogonal projections $z_{n,j}\in \mathcal{Z}(M)$. Choose $J$ large enough that
the set \[ F:=\Omega\setminus \cup\{ F_{n,j}: n,j\leq J\}\] has measure at most $\delta$,
and define a finitely nonzero column by $C=\sum_{n,j\leq J}z_{n,j}C_{n,j}$. Then $C^*C=\sum_{n,j\leq J}z_{n,j}$ and denote this sum by $z$, a central projection. If we extend $C$ to
$\tilde{C}$ by placing $1-z$ in a vacant position, then $\tilde{C}^*\tilde{C}=1$. Put $w=\tilde{C}^*X\tilde{C}\in W$.
We see that
\begin{align}
\psi(w)=\sum_{i=1}^r\sum_{n,j\leq J}&\int_{F_{n,j}}\langle C_{n,j}(\omega)^*X(\omega)C_{n,j}(\omega)\xi_i(\omega),\eta_i(\omega)
\rangle \,
d\mu(\omega)\notag\\
+\sum_{i=1}^r
&\int_F \langle (1-z)\xi_i(\omega),\eta_i(\omega)\rangle \, d\mu(\omega)
\end{align}\label{10.8}
so, from \eqref{10.7},
\begin{equation}\label{10.9}
{\mathrm{Re}}\,\psi(w)\leq |\psi(w)|\leq \delta rK^2\sum_{n,j\leq J}\mu(F_{n,j})+\mu(F) rK^2<2\delta rK^2= 2/3.
\end{equation}
This contradicts ${\mathrm{Re}}\,\psi(w)\geq 1$, proving the result.
\end{proof}

\begin{lem}\label{lem2.11}
Let $M\subseteq \B(\Hil)$ be a II$_1$ von Neumann algebra, let $x\in M$, and fix vectors $\xi_1,\ldots,\xi_r,\eta_1,\ldots,\eta_r\in \Hil$. There exist a separable subspace $\K\subseteq \Hil$ and a
von Neumann algebra $N\subseteq M$ with the following properties:
\begin{itemize}
\item[(i)]
$x\in N$,
\item[(ii)]
$\xi_1,\ldots,\xi_r,\eta_1,\ldots,\eta_r\in \K$,
\item[(iii)]
$\mathcal{Z}(N)=\mathcal{Z}(M)\cap N$,
\item[(iv)]
$\K$ is an invariant subspace for $N$.
\end{itemize}
\end{lem}

\begin{proof}
We will define inductively increasing sequences $A_n\subseteq M$ and $\K_n\subseteq \Hil$ of respectively separable unital C$^*$-algebras and separable subspaces with the following
properties:
\begin{itemize}
\item[(a)]
$x\in A_1$,
\item[(b)]
$\xi_1,\ldots,\xi_r,\eta_1,\ldots,\eta_r\in \K_1$,
\item[(c)]
for $a\in A_n$ and $\zeta\in \K_n$, $a\zeta\in \K_{n+1}$,
\item[(d)]
if $a\in A_n$, then $\mathbb{T}_M(a)\in A_{n+1}$, where $\mathbb{T}_M$ is the center-valued trace on $M$,
\item[(e)]
$A_{n+1}$ contains a countable set $\mathcal{U}_{n+1}\subseteq M$ of unitaries such that
\[\mathbb{T}_M(a)\in \overline{{\mathrm{conv}}}^{\|\cdot\|}
\{u^*au: u\in \mathcal{U}_{n+1}\},\ \ \ a\in A_n.\]
\end{itemize}

We define $A_1=C^*(x,1)$ and $\K_1=\overline{{\mathrm{span}}}\{\xi_1,\ldots,\xi_r,\eta_1,\ldots,\eta_r\}$.
For a countable norm dense subset $\mathcal{S}_1\subseteq A_1$, we choose a countable set $\mathcal{U}_2\subseteq M$ of unitaries so that
\[\mathbb{T}_M(a)\in \overline{{\mathrm{conv}}}^{\|\cdot\|}\{u^*au:u\in \mathcal{U}_2\}, \ \ \ a\in \mathcal{S}_1.\]
A simple norm approximation then shows that this also holds for $a\in A_1$. Now define $A_2=C^*\{a,\mathbb{T}_M(a), u: a\in A_1, u\in \mathcal{U}_2\}$ and
$\K_2=\overline{{\mathrm{span}}}\{a\zeta: a\in A_1,\zeta\in \K_1\}$. Then (a)-(e) hold by construction. If $A_n$ and $\K_n$ have been defined, then choose a countable set
$\mathcal{U}_{n+1}\subseteq M$ of unitaries so that (e) holds, and define
\[A_{n+1}=C^*\{a,\mathbb{T}_M(a),u:a\in A_n, u\in \mathcal{U}_{n+1}\}\]
and
\[\K_{n+1}=\overline{{\mathrm{span}}}\{a\zeta:a\in A_n,\zeta\in \K_n\}.\]
Then (a)-(e) hold, and we define $A$ to be the norm closure of $\cup_{n\geq 1}A_n$, $N$ to be the von Neumann algebra generated by $A$, and $\K$ to be the norm closure of $\cup_{n\geq
1}\K_n$. Then $\K$ is a separable subspace of $\Hil$ and (i), (ii) and (iv) are immediately seen to hold. It remains to verify (iii).

The inclusion $\mathcal{Z}(M)\cap N\subseteq \mathcal{Z}(N)$ is clear, so we only need to show the reverse containment. Given $\vp >0$ and $a\in A_n$, choose unitaries $u_1,\ldots,u_k\in
\mathcal{U}_{n+1}$ and non-negative constants $\lambda_i$ summing to 1 so that
\begin{equation}\label{11.1}
\left\|\mathbb{T}_M(a)-\sum_{i=1}^k\lambda_iu_i^*au_i\right\| <\vp,
\end{equation}
possible by (e). It follows that $\mathbb{T}_M(a)\in N$ and thus lies in $\mathcal{Z}(N)$. When we apply $\mathbb{T}_N$ to \eqref{11.1}, the result is
\begin{equation}\label{11.2}
\left\|\mathbb{T}_M(a)-\sum_{i=1}^k \lambda_i \mathbb{T}_N(a)\right\| < \vp.
\end{equation}
Since $\vp >0$ was arbitrary, this shows that $\mathbb{T}_M$ and $\mathbb{T}_N$ agree on $A$, and thus agree on $N$ since both maps are normal \cite[Theorem 8.2.8 (vi)]{KR2}. If $z\in \mathcal{Z}(N)$, then
$z=\mathbb{T}_N(z)=\mathbb{T}_M(z)\in \mathcal{Z}(M)\cap N$, as required.
\end{proof}

The next lemma is the general version of Lemma \ref{lem2.10} and removes the assumption of a separably acting von Neumann algebra.

\begin{lem}\label{lem2.12}
Let $M\subseteq \B(\Hil)$ be a II$_1$ von Neumann algebra and let $x\in M$ be an element for which there exists a unital contractive $\mathcal{Z}(M)$-bimodule map $\Phi:M\to
\mathcal{Z}(M)$ such that $\Phi(x)=0$. Let
\[W=\{C^*XC: C\in {\mathrm{Col}}_{n,1}(M), n\geq 1\}\subseteq M,\]
where $X={\mathrm{diag}}(x,x,x,\ldots)$. Then $0\in \overline{W}^{w^*}$.
\end{lem}

\begin{proof}
We argue by contradiction, so suppose that $0\notin \overline{W}^{w^*}$. As in Lemma \ref{lem2.10}, there exist vectors $\xi_1,\ldots,\xi_r,\eta_1,\ldots,\eta_r\in \Hil$ so that the functional
$\psi\in M_*$ defined by $\psi(\cdot)=\sum_{i=1}^r\langle \cdot \xi_i,\eta_i\rangle$ satisfies
\begin{equation}\label{12.1}
{\mathrm{Re}}\,\psi(w)\geq 1,\ \ \ w\in W.
\end{equation}

By Lemma \ref{lem2.11}, choose $N$ and $\K$ with the stated properties and let $p\in N'$ be the projection onto $\K$. Let $\Phi_1$ be the restriction of $\Phi$ to span$\{\mathcal{Z}(N),x\}$.
The range of $\Phi_1$ is $\mathcal{Z}(N)$ since $\mathcal{Z}(N)=\mathcal{Z}(M)\cap N$ and $\Phi(x)=0$. Thus there is a completely positive extension, also denoted by $\Phi_1$, so that
$\Phi_1:N\to \mathcal{Z}(N)$ is a unital $\mathcal{Z}(N)$-bimodule map and $\Phi_1(x)=0$. The map $n\mapsto np:N\to \B(\K)$ is a $*$-homomorphism and its kernel is a $w^*$-closed ideal
in $N$, so has the form $Nz_0$ for some projection $z_0\in \mathcal{Z}(N)\subseteq\mathcal{Z}(M)$. Since $z_0p=0$, we have the relations $z_0\xi_i=0$, $z_0\eta_i=0$, $(1-z_0)\xi_i=\xi_i$,
and $(1-z_0)\eta_i=\eta_i$ for $1 \leq i\leq r$. Moreover, $n\mapsto np$ is a faithful representation of $N(1-z_0)$ on $(1-z_0)\K$.

 For $C\in {\mathrm{Col}}_{n,1-z_0}(N(1-z_0))$, extend to $\tilde{C}\in {\mathrm{Col}}_{n+1,1}(N)$ by placing $z_0$ in the last position. Then
\begin{equation}\label{12.2}
{\mathrm{Re}}\,\sum_{i=1}^r\langle C^*(X(1-z_0))C\xi_i,\eta_i\rangle=
{\mathrm{Re}}\,\sum_{i=1}^r \langle \tilde{C}^*X\tilde{C}\xi_i,\eta_i\rangle \geq 1
\end{equation}
from \eqref{12.1}. This contradicts Lemma \ref{lem2.10} applied to $N(1-z_0)$ acting on the separable Hilbert space $(1-z_0)\K$ with distinguished element $x(1-z_0)$ for which
$\Phi_1(x(1-z_0))=0$. Thus $0\in \overline{W}^{w^*}$ as required.
\end{proof}

\section{Main crossed product results}\label{main}

In this section we now apply the lemmas of Section \ref{prelim} to obtain our main results pertaining to C$^*$-algebra crossed products, specifically Proposition \ref{average} and Theorem \ref{intermediate}.

\begin{pro}\label{average}
Let $A$ be a unital simple C$^*$-algebra and let $\alpha$ be an outer $*$-automorphism of $A$. Given $x\in A$ and $\delta 
> 0$, there exist $a_1,\ldots ,a_k\in A$ such that 
$\sum_{i=1}^k a_i^*a_i=1$ and
\begin{equation}\label{average.1}
\left\|\sum_{i=1}^k a_i^*x\alpha(a_i)\right\| < \delta.
\end{equation}
\end{pro}

\begin{proof}
We may assume that $A$ is in its universal representation on a Hilbert space $\Hil$ so that $A^{**}$ is identified with $A''$. Let $\alpha^{**}$ be the extension of $\alpha$ to a
$*$-automorphism of $A^{**}$. There exists an $\alpha^{**}$-invariant projection $z_1\in \Z(A^{**})$ so that the restriction of $\alpha^{**}$ to $A^{**}z_1$ is properly outer while the
restriction to $A^{**}(1-z_1)$ is inner and implemented by $u(1-z_1)$ for a unitary $u\in A^{**}$. This automorphism leaves invariant the type decomposition of $A^{**}(1-z_1)$ so we may
choose projections $z_2,z_3\in \Z(A^{**})$ so that $z_2+z_3=1-z_1$ and $A^{**}z_2$ is type II$_1$ while $A^{**}z_3$ has no type II$_1$ part. Simplicity of $A$ shows that $A$ has a
faithful represention $\pi_i$ as $Az_i$ on each $z_i\Hil$ for $1\leq i \leq 3$.

By Lemma \ref{propouter} applied to $\pi_1(A)$ acting on $z_1\Hil$, we see that $0\in \ovl{W_{\pi_1(A)}(\pi_1(x),\alpha)}^{w^*}$. For the representation $\pi_2$ of $A$ on $z_2\Hil$ where $\alpha$ is implemented by $uz_2$, we
apply Corollary \ref{Zmap2} to obtain a unital completely positive $\Z(A^{**}z_2)$-bimodular map $\Phi:A^{**}z_2\to \Z(A^{**})z_2$ such that $\Phi(xuz_2)=0$, and then we use Lemma
\ref{lem2.12} to deduce that $0\in\ovl{W_{\pi_2(A)}(\pi_2(x),\alpha)}^{w^*}$. We reach the same conclusion for the representation $\pi_3$ of $A$ on $z_3\Hil$ by appealing to Lemma \ref{no2-1}. Amalgamation of these 
statements on the three orthogonal subspaces that span $\Hil$ leads to $0\in \ovl{W_A(x,\alpha)}^{w^*}$, so $0\in \ovl{W_{A^{**}}(x,\alpha^{**})}^{w^*}$ by Lemma \ref{WtoN} (i). The second part of this lemma then shows that 
$0\in \ovl{W_A(x,\alpha)}^{\|\cdot\|}$ and the result now follows.
\end{proof}

If a discrete group $G$ acts on a C$^*$-algebra $A\subseteq \B(\Hil)$ by automorphisms $\{\alpha_g: g\in G\}$,  then the reduced C$^*$-crossed product $A\rtimes_{\alpha,r}G$ is constructed as follows (see \cite{VD}). We define a representation of $A$ on $\Hil\otimes \ell^2(G)$ by
\begin{equation}\label{rcp.1}
\pi_\alpha(a)(\xi\otimes \delta_t)=\alpha_{t^{-1}}(a)\xi\otimes \delta_t,\ \ \ \xi\in \Hil,\ a\in A,\ t\in G,
\end{equation}
and a unitary representation of $G$ by
\begin{equation}\label{rcp.2}
\lambda(g)(\xi\otimes \delta_t)=\xi\otimes \delta_{gt},\ \ \ \xi\in \Hil,\ g,t\in G.
\end{equation}
Then $A\rtimes_{\alpha,r}G$ is the C$^*$-algebra generated by these operators. If the representation of $A$ on $\Hil$ extends to a normal faithful representation of $A^{**}$ on $\Hil$, then the von Neumann crossed product $A^{**}\rtimes_{\alpha^{**}}G$ is the $w^*$-closure of $A\rtimes_{\alpha,r}G$. These operators satisfy
\begin{equation}\label{rcp.3}
\lambda(g)\pi_\alpha(a)\lambda(g)^*=\pi_\alpha(g(a)),\ \ \ a\in A,\ g\in G,
\end{equation}
showing that there is always a representation of $A$ for which the automorphisms $\alpha_g$ are spatially implemented by unitary operators $a_g$. If we take a faithful normal representation of $\mathrm{C}^*(A\cup\{a_g:g\in G\})^{**}$, then we may always assume that the automorphisms are spatially implemented, an important observation for the calculation in \eqref{unt.10}. To ease notation in this section, we will suppress $\lambda$ and $\pi_\alpha$, writing the generators of $A\rtimes_{\alpha,r}G$ simply as $a$ and $g$, subject to the condition
\begin{equation}\label{rcp.4}
gag^{-1}=\alpha_g(a),\ \ \ a\in A,\ g\in G.
\end{equation}
It will be necessary to return to the more cumbersome notation in Section \ref{tcp}.

There is a faithful normal conditional expectation $E:A^{**}\rtimes_{\alpha^{**}}G\to A^{**}$ (see \cite[p.365]{Tak1}) and this restricts to a conditional expectation of $A\rtimes_{\alpha,r}G$ onto $A$ which we also denote by $E$. Each element $x\in
A^{**}\rtimes_{\alpha^{**}}G$ has a Fourier series $\sum_{g\in G}x_gg$, and this converges in the $B$-topology as introduced by Bures in \cite{Bur} (see \cite[Section 3]{CaSm} for a detailed discussion of this topology). We define the {\it{support}} ${\mathrm{supp}}(x)$ to be the set of those $g\in G$ for which $x_g\ne 0$, equivalent to
the condition $E(xg^{-1})\ne 0$. For a subspace $X\subseteq A\rtimes_{\alpha,r}G$, we define ${\mathrm{supp}}(X)$ to be $\cup\{{\mathrm{supp}}(x):x\in X\}$. We say that $X$ is
{\emph{full}} if $X$ contains each element $y\in A\rtimes_{\alpha,r}G$ for which ${\mathrm{supp}}(y)\subseteq {\mathrm{supp}}(X)$.

The following lemma appears as Lemma 2.1 in \cite{CaSm} and we quote it here for the reader's convenience. The first three parts are due to Haagerup and Kraus \cite[Theorem 1.9]{HK}.

\begin{lem}\label{HKlem}
Let $G$ be a discrete  group with the $AP$ acting by $*$-automorphisms $\alpha_g$, $g\in G$, on a von Neumann algebra $M\subseteq B(H)$. Then there exist  a net
$(f_\gamma)_{\gamma\in\Gamma}$ of finitely supported functions on $G$ and a net $(T_\gamma:\cp\to\cp)_{\gamma\in\Gamma}$ of normal maps with the following properties:
\begin{itemize}
\item[\rm (i)]
For each $\gamma\in \Gamma$, $M_{f_\gamma}$ is completely bounded.
\item[\rm (ii)]
For each $g\in G$, $\underset{\gamma}{\lim}\,  f_\gamma(g)=1$.
\item[\rm (iii)]
For each function $h\in A(G)$, $\underset{\gamma}{\lim}\,  \|M_{f_\gamma}h-h\|=0$.
\item[\rm (iv)]
For each $\gamma\in\Gamma$, $T_\gamma$ is completely bounded.
\item[\rm (v)]
For each $y\in \cp$, $\underset{\gamma}{\lim}\,  T_\gamma(y)=y$ in the $w^*$-topology.
\item[\rm (vi)]
For each $x\in M$ and $g\in G$, $T_\gamma(\pi(x)\lambda_g)=f_\gamma(g)\pi(x)\lambda_g$.
\end{itemize}
\end{lem}

\begin{pro}\label{supp}
Let $A$ be a unital simple C$^*$-algebra and let $G$ be a discrete group acting on $A$ by outer $*$-automorphisms.
\begin{itemize}
\item[(i)] Let $X\subseteq A\rtimes_{\alpha,r}G$ be a norm closed $A$-bimodule and let $S\subseteq G$ be ${\mathrm{supp}}(X)$. Then
\begin{equation}\label{supp.1}
X_S:=\overline{{\mathrm{span}}}^{\|\cdot\|}\{Ag:g\in S\}
\end{equation}
is a norm closed $A$-bimodule and $X_S\subseteq X$.
\item[(ii)]
For each subset $S\subseteq G$, let
\begin{equation}\label{supp.2}
Y_S:=\{y\in A\rtimes_{\alpha,r}G:{\mathrm{supp}}(y)\subseteq S\}.
\end{equation}
Then the map $S\mapsto Y_S$ gives a bijection between the subsets of $G$ and the set of full $A$-bimodules in $A\rtimes_{\alpha,r}G$.
\item[(iii)]
If $G$ has the approximation property, then each norm closed $A$-bimodule $X\subseteq A\rtimes_{\alpha,r}G$ is full and has the form $X_S$ where $S={\mathrm{supp}}(X)$. Moreover, $X_S=Y_S$.
\end{itemize}
\end{pro}

\begin{proof}
(i) \quad Let $X\subseteq A\rtimes_{\alpha,r}G$ be a norm closed $A$-bimodule and let $S\subseteq G$ be the support of $X$. We first show that 
$X_S\subseteq X$.

Fix $g_0\in S$ and choose $x_0\in X$ so that $E(x_0g_0^{-1})\ne 0$. By simplicity of $A$, there exist elements $b_1,\ldots,b_n,c_1,\ldots,c_n\in A$ so that
$\sum_{i=1}^n b_iE(x_0g_0^{-1})c_i=1$. If we replace $x_0$ by $\sum_{i=1}^n b_ix_0\alpha_{g_0}^{-1}(c_i)\in X$, then we may assume that $E(x_0g_0^{-1})=1$.

Let $Y={\mathrm{span}}\{Ag: g\in G\}$, a norm dense subspace of $A\rtimes_{\alpha,r}G$. Given $\vp\in (0,1)$, consider the set $\mathcal P$ of pairs $(x,y)$ where $x\in X$,
$E(xg_0^{-1})=1$, $y\in Y$ and $\|x-y\|<\vp$. The existence of $x_0$ above shows that $\mathcal{P}$ is nonempty, so we may choose a pair $(x,y)\in \mathcal P$ for which the sum $y=\sum_{g\in F}y_gg$ (where $F\subseteq G$ is finite) has minimal length. Then
\begin{equation}\label{supp.3}
\|1-y_{g_0}\|=\|E((x-y)g_0^{-1})\|\leq \|x-y\|<\vp <1,
\end{equation}
so $y_{g_0}\ne 0$. Now suppose that there exists $g_1\ne g_0$ so that $y_{g_1}\ne 0$. We will show that this contradicts the minimal length in the selection of the pair $(x,y)$.

Choose $\delta >0$ so that $\|x-y\|+\delta <\vp$. Noting that $\alpha_{g_1g_0^{-1}}$ is outer on $A$, apply Proposition \ref{average} to obtain elements $a_1,\ldots,a_k\in A$ so that $\sum_{i=1}^k
a_i^*a_i=1$ and
\begin{equation}\label{supp.4}
\left\|\sum_{i=1}^k a_i^*y_{g_1}\alpha_{g_1g_0^{-1}}(a_i)\right\| < \delta.
\end{equation}
Define a complete contraction $\Psi:A\rtimes_{\alpha,r}G\to A\rtimes_{\alpha,r}G$ by
\begin{equation}\label{supp.5}
\Psi(z)=\sum_{i=1}^k a_i^* z \alpha_{g_0}^{-1}(a_i),\ \ \ z\in A\rtimes_{\alpha,r}G,
\end{equation}
and note that $\Psi$ maps $X$ to itself. The $g_0$-coefficient of $\Psi(x)$ remains 1, and the $g_1$-coefficient of $\Psi(y)$ is $\sum_{i=1}^k a_i^*y_{g_1}\alpha_{g_1g_0^{-1}}(a_i)$ so
has norm at most $\delta$ from \eqref{supp.4}. Write $\Psi(y)=\sum_{g\in F}\hat{y}_gg$. Then $\|\hat{y}_{g_1}\|<\delta$, so
\begin{equation}\label{supp.6}
\|\Psi(x)-(\Psi(y)-\hat{y}_{g_1}g_1)\|\leq \|\Psi(x-y)\|+\delta\leq \|x-y\|+\delta<\vp.
\end{equation}
Thus $(\Psi(x),\Psi(y)-\hat{y}_{g_1}g_1)\in \mathcal P$ and the second entry is a finite sum that is strictly shorter than the sum defining $y$ since its $g_1$-coefficient is now 0. This contradicts the minimal choice of
$(x,y)$, and so we conclude that $y_{g_1}=0$ for $g_1\ne g_0$, and hence that $y=y_{g_0}g_0$.

If we now apply this argument to successive choices $\vp=2^{-n}$ for $n\geq 1$, then we obtain sequences $\{x_n\}_{n=1}^\infty$ from $X$ and $\{y_n\}_{n=1}^\infty$ from $A$ so that
$E(x_ng_0^{-1})=1$ and $\|x_n-y_ng_0\|<2^{-n}$. Then $1-y_n=E(x_ng_0^{-1}-y_n)$, so $\lim_{n\to\infty}\|1-y_n\|=0$, and also $\lim_{n\to\infty}\|x_n-g_0\|=0$. We conclude that $g_0\in X$ and
so $Ag_0\subseteq X$. Since $g_0\in S$ was arbitrary, it follows that $X_S\subseteq X$, proving (i).

\medskip

\noindent (ii)\quad This part is immediate from the definition of fullness.

\medskip

\noindent (iii)\quad Let $X\subseteq A\rtimes_{\alpha,r}G$ be a norm closed $A$-bimodule and let $S={\mathrm{supp}}(X)$. By (i) and the definition of $Y_S$, we have
\begin{equation}\label{supp.7}
X_S\subseteq X\subseteq Y_S.
\end{equation}
 From Lemma \ref{HKlem} applied to $M=A^{**}$, we obtain a net 
of completely bounded maps $\{T_\gamma :A^{**}\rtimes_{\alpha^{**}}G\to A^{**}\rtimes_{\alpha^{**}}G\}_{\gamma\in\Gamma}$ converging in the point $w^*$-topology to the identity. As in
\cite[Theorem 1.9]{HK}, we also may assume that convergence is in the point norm topology on $A\rtimes_{\alpha,r}G$. Since each $T_\gamma$ arises from a finitely supported multiplier, we
see that $T_\gamma$ maps $Y_S$ into $X_S$ and point norm convergence to the identity establishes equality of the three spaces in \eqref{supp.7}. Since $X=Y_S$, fullness of $X$ follows from (ii).
\end{proof}

\begin{rem}\label{equalmod}
It may be that the $A$-bimodules defined in \eqref{supp.1} and \eqref{supp.2} are always equal. We know of no instance where $X_S \ne Y_S$, and any counterexample would require a group without the approximation property.
\end{rem}

\begin{thm}\label{intermediate}
Let $A$ be a unital simple C$^*$-algebra and let $G$ be a discrete group acting on $A$ by outer $*$-automorphisms.
There is a bijective correspondence between subgroups $H$ of $G$ and C$^*$-algebras $B$ satisfying $A\subseteq B\subseteq A\rtimes_{\alpha,r}G$ given by
\begin{equation}\label{intermediate.1}
H\mapsto A\rtimes_{\alpha,r}H.
\end{equation}
\end{thm}

\begin{proof}
Injectivity of the map $H\mapsto A\rtimes_{\alpha,r}H$ follows from Proposition \ref{supp} (ii) and so it suffices to prove that it is surjective. Fix an intermediate C$^*$-algebra $B$
and let $H\subseteq G$ be its support. By Proposition \ref{supp} (i), $B$ contains each element of $H$ and so $H$ is a subgroup of $G$. Let $E_H$ be the normal conditional expectation of
$A^{**}\rtimes_{\alpha^{**}}G$ onto $A^{**}\rtimes_{\alpha^{**}}H$ given by ${\mathrm{Ad}}W^*\circ (I\otimes E_{L(H)})\circ{\mathrm{Ad}}W$ where $E_{L(H)}$ is the conditional expectation of $L(G)$
onto $L(H)$ and $W$ is the operator of \cite[(2.4)]{CaSm}. By \cite[Lemma 3.1 (ii)]{CaSm2}, $E_H$ is continuous in the $B$-topology (this lemma does not require proper outerness of the
action on $A^{**}$). Thus, for each $x\in A^{**}\rtimes_{\alpha^{**}}G$ with Fourier series $x=\sum_{g\in G}x_gg$, we have
\begin{equation}\label{intermediate.2}
E_H(x)=\sum_{h\in H}x_hh.
\end{equation}
In particular $E_H$ maps $A\rtimes_{\alpha,r}G$ onto $A\rtimes_{\alpha,r}H$ and also acts as the identity on $Y_H$ as defined in Proposition \ref{supp} (ii). Since $B\subseteq Y_H$, we
see that $E_H$ acts as the identity on $B$, establishing that $B\subseteq A\rtimes_{\alpha,r}H$. The reverse containment follows from Proposition \ref{supp} (i), and this proves that
every intermediate C$^*$-subalgebra has the form $A\rtimes_{\alpha,r}H$ for a subgroup $H\subseteq G$.
\end{proof}

We conclude this section with some remarks on consequences of the last two results.
\begin{rem}\label{conseq1}
(i) \quad In \cite{Ki}, Kishimoto showed that crossed products of unital simple C$^*$-algebras by outer actions of discrete groups are again simple, following work of Elliott \cite{Ell} in the $AF$-algebra case. We can recapture this result as follows. If $J\subseteq A\rtimes_{\alpha,r}G$ is a nonzero norm closed ideal then $J$ is an $A$-bimodule so contains a group element by Proposition \ref{supp} (i). Thus $J=A\rtimes_{\alpha,r}G$, and simplicity of the crossed product follows.

\medskip

\noindent (ii) \quad If $A$ is a finitely generated unital simple C$^*$-algebra and $G$ is a countable discrete group acting on $A$ by outer $*$-automorphisms, then $A\rtimes_{\alpha,r}G$ is also finitely generated. If $\{g_1,g_2,\ldots\}$ is a listing of the elements of $G$ and $A$ is generated by $\{a_1,\ldots,a_k\}$, then the C$^*$-algebra $B$ generated by $\{a_1,\ldots,a_k,b\}$ with
$b=\sum_{i=1}^\infty g_i/2^i$ is intermediate with support $G$. By Theorem \ref{intermediate}, the only possibility is $B=A\rtimes_{\alpha,r}G$, showing that the crossed product is finitely generated.

\medskip

\noindent (iii) \quad If $A$ is a unital simple nuclear C$^*$-algebra and $G$ is a discrete amenable group, then all intermediate C$^*$-algebras are nuclear by Theorem \ref{intermediate}. This is in contrast to the fact that many nuclear C$^*$-algebras have non-nuclear subalgebras.

\medskip

\noindent (iv) \quad The assumption of an outer action in Theorem \ref{intermediate} seems essential. If $\mathbb{Z}$ acts trivially on a unital simple C$^*$-algebra $A$, then $A\rtimes_{\alpha,r}
\mathbb{Z}$ is $A\otimes_{\mathrm{min}}C(\mathbb{T})$ which can be identified with the C$^*$-algebra $C(\mathbb{T},A)$ of continuous $A$-valued functions on the circle $\mathbb{T}$.
This has a C$^*$-subalgebra $B$ consisting of those functions satisfying $f(1)=f(-1)$. If $Y$ is the figure of eight obtained 
from $\mathbb{T}$ by identifying the points $\pm 1$, then $B$ is isomorphic to
$A\otimes_{\mathrm{min}}C(Y)$. On the other hand, any nontrivial subgroup $H$ of $\mathbb{Z}$ is isomorphic to $\mathbb{Z}$, giving an isomorphism of $A\rtimes_{\alpha,r}H$ with
$A\rtimes_{\alpha,r}\mathbb{Z}$.
Thus $B$ cannot have the form $A\rtimes_{\alpha,r}H$, since this would lead to a topological isomorphism of $\mathbb{T}$ with $Y$.

\medskip

\noindent (v) \quad When $A$ is not simple the conclusion of Theorem
\ref{intermediate} may fail to hold. Let $A_1$ be a C$^*$-algebra that
admits a period 2 outer automorphism $\theta$ (for example,
C$^*_r(\mathbb{F}_2)$) and let $A=A_1\oplus A_1$ with a central projection
$z=(1,0)$. Define an outer automorphism $\alpha$ of $A$ by $\alpha(x,y)
=(\theta(x),\theta(y))$ for $x,y\in A_1$, which gives an outer action of
$\mathbb{Z}/2\mathbb{Z}$ on $A$. The intermediate C$^*$-algebra
$B:=A+Az\alpha$ does not have the form
$A\rtimes_{\alpha,r}\mathbb{Z}/2\mathbb{Z}$ for a subgroup $H\subseteq
\mathbb{Z}/2\mathbb{Z}$ since there are no nontrivial subgroups.
\end{rem}

\begin{rem}\label{conseq2}
(i) \quad Recall from \cite{PSS} that if we have an inclusion $B\subseteq A$ of C$^*$-algebras, then $B$ is said to norm $A$ when the following is satisfied: for any integer $k$ and
matrix $X\in \mathbb{M}_k(A)$,
\begin{equation}
\|X\|=\sup \|RXC\|
\end{equation}
where the supremum is taken over row matrices $R$ and column matrices $C$ of length $k$ with entries from $B$ and satisfying $\|C\|,\|R\|\leq 1$. Since $RXC\in A$, the
point of this definition is to reduce the calculation of norms in $\mathbb{M}_k(A)$ to that of norms in $A$, and the concept of norming has proved useful in showing
complete boundedness of certain types of bounded maps. We use this below.

We showed in \cite[Theorem 5.1]{CaSm2} that if a discrete group acts on a von Neumann algebra $M$ by properly outer automorphisms, then $M$ norms $M\rtimes_\alpha G$. An examination of the proof shows that the ``properly outer'' hypothesis is only necessary when $M$ has a type I$_{\mathrm{finite}}$ central summand, and so the result is true for general actions when $M$ has no finite dimensional representations. This is the case for the second dual of an infinite dimensional simple C$^*$-algebra $A$, and so $A^{**}\rtimes_{\alpha^{**}} G$ is normed by $A^{**}$ for an arbitrary action of $G$ on $A$. Then \cite[Lemma 2.3 (ii)]{PSS} shows that $A$ norms 
$A^{**}\rtimes_{\alpha^{**}} G$, so also norms any subspace of $A\rtimes_{\alpha,r} G$.

\medskip

\noindent (ii) \quad If $A$ is a unital C$^*$-algebra contained in an $A$-bimodule $X$, then a map $\theta:X\to X$ is called an $A$-bimodule map if $\theta|_A$ is a $*$-automorphism and 
\begin{equation}
\theta(a_1xa_2)=\theta(a_1)\theta(x)\theta(a_2),\ \ \ x\in X,\ a_1,a_2\in A.
\end{equation}
If $\theta$ is isometric and surjective, then a natural question is whether $\theta$ extends to a $*$-automorphism of C$^*(X)$, or perhaps of W$^*(X)$. Mercer \cite{MerC,Mer} was the first to consider such a problem, and results of this type are now called Mercer's theorem in the literature (see \cite{CPZ,CaSm,CaSm2} for later versions). 

In the case that $A$ is a unital simple C$^*$-algebra and $G$ is a discrete group acting by outer automorphisms, consider a norm closed $A$-bimodule $X$ satisfying $A\subseteq X\subseteq A\rtimes_{\alpha,r}G$ and an isometric surjective $A$-bimodule map $\theta$ on $X$.
Then C$^*(X)$ is $A\rtimes_{\alpha,r}H$ for a subgroup $H$ of $G$ by Theorem \ref{intermediate} and so is simple by Remark \ref{conseq1} (i). Thus C$^*(X)$ is also the C$^*$-envelope C$^*_{env}(X)$, and $\theta$ is completely isometric as in \cite[Lemma 6.2]{CaSm2} since $A$ norms $A\rtimes_{\alpha,r}G$ by (i). Then $\theta$ extends to a $*$-automorphism of C$^*(X)$; the details of the argument are in the proof of \cite[Theorem 6.6]{CaSm2} and we do not repeat them here.
\end{rem}

\section{Twisted crossed products}\label{tcp}

In this section, we extend the results of Section \ref{main} to the case of reduced twisted crossed products. Some of the arguments are essentially the same, so we will concentrate on
those points where significant differences arise. We thank Iain Raeburn for his helpful guidance through the literature of twisted crossed products.

Following the treatment in \cite{BK}, a twisted dynamical system is a quadruple $(A,G,\alpha,\sigma)$ where $A$ is a C$^*$-algebra (always assumed to be unital), $G$ is a discrete group,
and $\alpha:G\to {\mathrm{Aut}}(A)$ and $\sigma:G\times G\to \mathcal{U}(A)$ are maps satisfying
\begin{equation}\label{tw.1}
\alpha_s\alpha_t=\Ad\sigma(s,t)\circ \alpha_{st},\ \ \ s,t\in G,
\end{equation}
\begin{equation}\label{tw.2}
\alpha_r(\sigma(s,t))\sigma(r,st)=
\sigma(r,s)\sigma(rs,t),\ \ \ r,s,t\in G,
\end{equation}
and
\begin{equation}\label{tw.2.5}
\sigma(e,s)=\sigma(s,e)=1,\ \ \ s\in G.
\end{equation}
The map $\sigma$ is called a cocycle, and \eqref{tw.2} arises by applying \eqref{tw.1} to $\alpha_r\circ(\alpha_s\circ\alpha_t)=
(\alpha_r\circ\alpha_s)\circ\alpha_t$.

Assume that $A$ is faithfully represented on a Hilbert space $\Hil$. We define a representation $\pi_\alpha$ of $A$ on $\Hil\otimes \ell^2(G)$ and a map $\lambda_\sigma:G\to
\B(\Hil\otimes \ell^2(G))$ as follows:
\begin{equation}\label{tw.3}
\pi_\alpha(a)(\xi\otimes \delta_g)=\alpha_{g^{-1}}(a)\xi\otimes \delta_g,\ \ \ a\in A,\ \xi\in \Hil,\ g\in G,
\end{equation}
and
\begin{equation}\label{tw.4}
\lambda_\sigma(s)(\xi \otimes \delta_t)=\sigma(t^{-1}s^{-1},s)\xi\otimes\delta_{st},
\ \ \ \xi\in\Hil,\ s,t\in G.
\end{equation}
As noted in \cite{BK}, these satisfy
\begin{equation}\label{tw.5}
\pi_\alpha(\alpha_s(a))=\lambda_\sigma(s)\pi_\alpha(a)
\lambda_\sigma(s)^*,\ \ \ a\in A,\ s\in G,
\end{equation}
and
\begin{equation}\label{tw.6}
\lambda_\sigma(s)\lambda_\sigma(t)=\pi_\alpha(\sigma(s,t))\lambda_\sigma(st),
\ \ \ s,t\in G.
\end{equation}
The {\it{reduced twisted crossed product}} $A\rtimes^\sigma_{\alpha,r}G\subseteq \B(\Hil\otimes \ell^2(G))$ is defined to be the C$^*$-algebra generated by the operators $\pi_\alpha(a)$ and
$\lambda_\sigma(g)$ for $a\in A$ and $g\in G$. As shown in \cite{Q}, it is independent of the choice of the faithful representation of $A$. When the cocycle is trivial ($\sigma(s,t)=1$),
this is just the usual reduced C$^*$-crossed product.
If we replace $A$ by a von Neumann algebra $M$, then the $w^*$-closure of the algebra generated by these operators is the twisted von Neumann algebra, which we denote by
$M\rtimes^\sigma_{\alpha,vn}G$ to avoid confusion with the full twisted crossed product \cite{PR}, usually written as $A\rtimes^\sigma_\alpha G$.

As was mentioned in the introduction,  twisted crossed products arise when expressing $A\rtimes_{\alpha,r}G$ as $(A\rtimes_{\alpha,r}N)\rtimes_{\beta,r}^\sigma G/N$, where $N$ is a normal subgroup of $G$. We digress briefly to show how this is achieved (see \cite{Bed}).
Specify a cross section $\phi:G/N\to G$, and define $\beta_s$ to be $\Ad \phi(s)$ for $s\in G/N$, noting that $\beta_s$ maps $A\rtimes_{\alpha,r}N$ to itself since $N$ is a normal subgroup. The cocycle $\sigma$ is then given by
\begin{equation}\label{tw.6.5}
\sigma(s,t):=\phi(s)\phi(t)\phi(st)^{-1}\in N\subseteq \mathcal{U}(A\rtimes_{\alpha,r}N),\ \ \ \ \ s,t\in G/N,
\end{equation}
and $\sigma$ measures the extent to which $\phi$ fails to be a group homomorphism.

Let $(A,G,\alpha,\sigma)$ and $(A,G,\beta,\mu)$ be two twisted dynamical systems. As in \cite{PR}, we say that these are {\it{exterior equivalent}} if there is a map
$v:G\to{\mathcal{U}}(A)$ so that
\begin{equation}\label{tw.7}
\beta_s=\Ad v_s\circ \alpha_s,\ \ \ s\in G,
\end{equation}
and
\begin{equation}\label{tw.8}
\mu(s,t)=v_s\alpha_s(v_t)\sigma(s,t)v_{st}^*,\ \ \ s,t\in G.
\end{equation}
Exterior equivalence was related to isomorphism of the full twisted crossed products in \cite[Lemma 3.3]{PR}. Our next result is the analogous statement for the reduced case.

\begin{lem}\label{exterior}
Let $A$ be a unital C$^*$-algebra faithfully represented on a Hilbert space $\Hil$ and let $G$ be a discrete group. Let $(A,G,\alpha,\sigma)$ and $(A,G,\beta,\mu)$ be twisted dynamical
systems, and suppose that they are exterior equivalent by a map $v:G\to \mathcal{U}(A)$. Then there exists a unitary operator $V$ on $\Hil\otimes \ell^2(G)$ so that $\Ad V$is a spatial
isomorphism of $A\rtimes^\sigma_{\alpha,r}G$ onto $A\rtimes^\mu_{\beta,r}G$ and maps $\pi_\alpha(a)\lambda_\sigma(g)$ to $\pi_\beta(a)\pi_\beta(v_g^*)\lambda_\mu(g)$ for $a\in A$ and
$g\in G$.

Moreover, if $A^{**}$ is faithfully normally represented on $\Hil$, then $\Ad V$ is a spatial isomorphism of $A^{**}\rtimes^\sigma_{\alpha^{**},vn}G$ onto
$A^{**}\rtimes^\mu_{\beta^{**},vn}G$.
\end{lem}

\begin{proof}
As noted above, $A\rtimes^\sigma_{\alpha,r}G$ is independent of the faithful representation of $A$, so we may assume at the outset that we have faithfully normally represented $A^{**}$
on a Hilbert space $\Hil$.

Define a unitary $V$ on $\Hil\otimes\ell^2(G)$ by
\begin{equation}\label{ext.1}
V(\xi\otimes \delta_t)=v_{t^{-1}}\xi\otimes\delta_t,\ \ \ \xi\in\Hil,\ t\in G.
\end{equation}
For $a\in A$, $\xi\in\Hil$ and $t\in G$,
\begin{align}
V\pi_\alpha(a)V^*(\xi\otimes \delta_t)&=V\pi_\alpha(a)(v_{t^{-1}}^*\xi\otimes\delta_t)
=V(\alpha_{t^{-1}}(a)v_{t^{-1}}^*\xi\otimes\delta_t)\notag\\
&=v_{t^{-1}}\alpha_{t^{-1}}(a)v_{t^{-1}}^*\xi\otimes\delta_t
=\beta_{t^{-1}}(a)\xi\otimes\delta_t\notag\\
&=\pi_\beta(a)(\xi\otimes\delta_t),\label{ext.2}
\end{align}
where we have used \eqref{tw.7} in the penultimate equality.
Since $\xi\in \Hil$ and $t\in G$ were arbitrary, we see that $\Ad V(\pi_\alpha(a))=\pi_\beta(a)$.

The change of variables $(s,t)\mapsto (t^{-1}g^{-1},g)$ in \eqref{tw.8} gives
\begin{equation}\label{ext.3}
\mu(t^{-1}g^{-1},g)=v_{t^{-1}g^{-1}}\alpha_{t^{-1}g^{-1}}(v_g)
\sigma(t^{-1}g^{-1},g)v_{t^{-1}}^*
\end{equation}
so a rearrangement of the terms in \eqref{ext.3} leads to
\begin{equation}\label{ext.4}
\sigma(t^{-1}g^{-1},g)v_{t^{-1}}^*\mu(t^{-1}g^{-1},g)^*=
\alpha_{t^{-1}g^{-1}}(v_g^*)v_{t^{-1}g^{-1}}^*.
\end{equation}
Now, for $\xi\in\Hil$ and $g,t\in G$,
\begin{align}
\Ad V(\lambda_\sigma(g))(\xi\otimes\delta_t)&=V\lambda_\sigma(g)(v_{t^{-1}}^*\xi
\otimes\delta_t)\notag\\
&=V(\sigma(t^{-1}g^{-1},g)v_{t^{-1}}^*\xi\otimes\delta_{gt})
\ \ \ \text{ (from \eqref{tw.4})}\notag\\
&=v_{t^{-1}g^{-1}}\sigma(t^{-1}g^{-1},g)v_{t^{-1}}^*\xi\otimes\delta_{gt}
 \notag\\
&=v_{t^{-1}g^{-1}}\sigma(t^{-1}g^{-1},g)v_{t^{-1}}^*
\mu(t^{-1}g^{-1},g)^*\mu(t^{-1}g^{-1},g)\xi\otimes\delta_{gt}\notag\\
&=v_{t^{-1}g^{-1}}\alpha_{t^{-1}g^{-1}}(v_g^*)v_{t^{-1}g^{-1}}^*
\mu(t^{-1}g^{-1},g)\xi\otimes\delta_{gt}\ \ \ \text{ (from \eqref{ext.4})}\notag\\
&= \beta_{t^{-1}g^{-1}}(v_g^*)\mu(t^{-1}g^{-1},g)\xi\otimes \delta_{gt} \ \ \ \text{ (from \eqref{tw.7})}.\label{ext.5}
\end{align}
On the other hand,
\begin{align}
\pi_\beta(v_g^*)\lambda_\mu(g)(\xi\otimes \delta_t)&=
\pi_\beta(v_g^*)(\mu(t^{-1}g^{-1},g)\xi\otimes\delta_{gt})\notag\\
&=\beta_{t^{-1}g^{-1}}(v_g^*)\mu(t^{-1}g^{-1},g)\xi\otimes\delta_{gt},\label{ext.6}
\end{align}
for $\xi\in\Hil$, $g,t\in G$, and so \eqref{ext.5} and \eqref{ext.6} establish that
\begin{equation}\label{ext.7}
\Ad V(\lambda_\sigma(g))=\pi_\beta(v_g^*)\lambda_\mu(g),
\ \ \ g\in G.
\end{equation}
Thus $\Ad V$ gives a spatial isomorphism between the two reduced twisted crossed products with the stated properties in both the C$^*$-algebra and von Neumann algebra cases.
\end{proof}

In the next result, we show that a reduced twisted crossed product can be untwisted at the von Neumann algebra level by tensoring with a copy of $\B(\ell^2(G))$. This is inspired by a
similar result of Sutherland \cite[Theorem 5.1]{Sut} who proved this for factors and scalar valued cocycles. Our proof follows that of \cite[Theorem 3.4]{PR} where the untwisting is
accomplished for full twisted crossed products by tensoring with the C$^*$-algebra of compact operators on $\ell^2(G)$ when considering only discrete groups ($L^2(G)$ for general locally
compact groups in \cite{PR}).

We denote by $\lam_g$  the unitary operator on $\ell^2(G)$  defined by
$\lam_g(\delta_t)=
\delta_{gt}$, for $g,t\in G$. 

\begin{lem}\label{untwist} Let $A$ be a unital C$^*$-algebra with $A^{**}$ faithfully normally represented on a Hilbert space $\Hil$, let $G$ be a discrete group, and let
$(A,G,\alpha,\sigma)$ be a twisted dynamical system. There exist an action $\beta$ of $G$ on $A^{**}\vntensor \B(\ell^2(G))$, a map $v:G\to \mathcal{U}(A^{**}\vntensor \B(\ell^2(G)))$,
and a surjective spatial isomorphism
\begin{equation}\label{unt.1}
\phi:(A^{**}\vntensor \B(\ell^2(G)))\rtimes^{\sigma\otimes 1}_{\alpha^{**}\otimes id, vn}G \to
(A^{**}\vntensor \B(\ell^2(G)))\rtimes_{\beta,vn}G
\end{equation}
satisfying
\begin{equation}\label{unt.2}
\phi(\pi_{\alpha^{**}\otimes id}(x))=\pi_\beta(x),
\ \ \ x\in A^{**}\vntensor \B(\ell^2(G)),
\end{equation}
and
\begin{equation}\label{unt.3}
\phi(\lambda_{\sigma\otimes 1}(g))=\pi_\beta(v_g^*)(1\otimes \lam_g),\ \ \ g\in G.
\end{equation}
\end{lem}

\begin{proof}
From the discussion after equation \eqref{rcp.3}, we will assume that the automorphism $\alpha_g$ is spatially implemented by a unitary $a_g\in \B(\Hil)$ for $g\in G$.

We follow the construction in the proof of \cite[Theorem 3.4]{PR}. For $s\in G$, we define a unitary $M_s$ on $\Hil\otimes \ell^2(G)$ by
\begin{equation}\label{unt.4}
M_s(\xi\otimes\delta_t)=\sigma(s,t)\xi\otimes\delta_t,
\ \ \ \xi\in\Hil,\ t\in G.
\end{equation}
If $z\in (A^{**})'$, then $z$ commutes with $\sigma(s,t)$ and so $z\otimes 1$ commutes with $M_s$. Thus $M_s\in ((A^{**})'\otimes 1)'=A^{**}\vntensor \B(\ell^2(G))$, and it follows that
\begin{equation}\label{unt.5}
\beta_s:=\Ad[(1\otimes \lam_s)M_s^*]\circ(\alpha_s^{**}\otimes \mathrm{id})
\end{equation}
is an automorphism of $A^{**}\vntensor \B(\ell^2(G))$ for each $s\in G$. We now check that $\beta_s\beta_t=\beta_{st}$ for $s,t\in G$, and it suffices to do this on generators $x\otimes
E_{p,q}$, where $x\in A^{**}$ and $E_{p,q}$ is the rank one partial isometry that maps $\delta_q$ to $\delta_p$. We denote the Kronecker delta function by $\Delta_{p,q}$. Then, using
\eqref{unt.5}, we compute that
\begin{align}
\beta_s(x\otimes E_{p,q})(\xi\otimes\delta_t)&=(1\otimes\lam_s)M_s^*(\alpha_s^{**}(x)\otimes E_{p,q})M_s (1\otimes \lam_{s^{-1}})(\xi\otimes\delta_t)\notag\\
&=(1\otimes\lam_s)M_s^*(\alpha_s^{**}(x)\otimes E_{p,q})M_s(\xi\otimes\delta_{s^{-1}t})\notag\\
&=(1\otimes \lam_s)M_s^*(\alpha_s^{**}(x)\otimes E_{p,q})(\sigma(s,s^{-1}t)\xi\otimes\delta_{s^{-1}t})\notag\\
&=\Delta_{q,s^{-1}t}(1\otimes\lam_s)M_s^*
(\alpha_s^{**}(x)\sigma(s,s^{-1}t)\xi\otimes\delta_p)\notag\\
&=\Delta_{sq,t}(1\otimes\lam_s)(\sigma(s,p)^*\alpha_s^{**}(x)
\sigma(s,s^{-1}t)\xi\otimes\delta_p)\notag\\
&=\Delta_{sq,t}\sigma(s,p)^*\alpha_s^{**}(x)\sigma(s,q)\xi\otimes\delta_{sp}
\notag\\
&=\sigma(s,p)^*\alpha_s^{**}(x)\sigma(s,q)\xi\otimes E_{sp,sq}(\delta_t),\label{unt.6}
\end{align}
for $x\in A^{**}$ and $p,q,s,t\in G$.
Thus
\begin{equation}\label{unt.7}
\beta_s(x\otimes E_{p,q})=\sigma(s,p)^*\alpha_s^{**}(x)\sigma(s,q)\otimes E_{sp,sq}.
\end{equation}
It follows that, for $x\in A^{**}$ and $p,q,s,t\in G$,
\begin{align}
\beta_s\beta_t(x\otimes E_{p,q})&=\beta_s(\sigma(t,p)^*\alpha_t^{**}(x)\sigma(t,q)\otimes E_{tp,tq})\notag\\
&=\sigma(s,tp)^*\alpha_s^{**}(\sigma(t,p)^*)\alpha_{s}^{**}(\alpha_t^{**}(x))\alpha_s^{**}(\sigma(t,q))\sigma(s,tq)
\otimes E_{stp,stq}\notag\\
&=\sigma(s,tp)^*\alpha_s^{**}(\sigma(t,p)^*)\sigma(s,t)\alpha_{st}^{**}(x)\sigma(s,t)^*\alpha_s^{**}(\sigma(t,q))\sigma(s,tq)\otimes E_{stp,stq}\notag\\
&=\sigma(st,p)^*\alpha_{st}^{**}(x)\sigma(st,q)\otimes E_{stp,stq}\notag\\
&=\beta_{st}(x\otimes E_{p,q}),\label{unt.8}
\end{align}
where we have used \eqref{tw.1} in the third equality, \eqref{tw.2} twice in the fourth, and \eqref{unt.7} in the last. Thus 
$\beta_s\beta_t=\beta_{st}$ and so $\beta$ is an action of $G$ on $A^{**}\vntensor \B(\ell^2(G))$.

We now show that the two dynamical systems $(A^{**}\vntensor \B(\ell^2(G)),G,\alpha^{**}\otimes {\mathrm{id}},\sigma\otimes 1)$ and  $(A^{**}\vntensor \B(\ell^2(G)),G,\beta,1)$ are exterior equivalent. Define a map
$v:G\to \mathcal{U}(A^{**}\vntensor \B(\ell^2(G)))$ by
\begin{equation}\label{unt.9}
v_s=(1\otimes\lam_s)M_s^*,\ \ \ s\in G.
\end{equation}
Both operators in this product commute with $(A^{**})'\otimes 1$, so $v_s\in A^{**}\vntensor \B(\ell^2(G))$. By construction, $\beta_s=\Ad v_s\circ (\alpha_s\otimes {\mathrm{id}})$ so
\eqref{tw.7} is satisfied. Recalling that $\alpha_s^{**}$ is spatially implemented by a unitary $a_s\in \B(\Hil)$ for $s\in G$, we see that
\begin{align}
(\alpha_s^{**}\otimes{\mathrm{id}})(v_t)(\xi \otimes \delta_r)&=(a_s\otimes 1)v_t(a_s^*\xi\otimes\delta_r)\notag\\
&=(a_s\otimes 1)(1\otimes \lam_t)(\sigma(t,r)^*a_s^*\xi\otimes\delta_r)\notag\\
&=a_s\sigma(t,r)^*a_s^*\xi\otimes\delta_{tr}
=\alpha_s^{**}(\sigma(t,r)^*)\xi\otimes\delta_{tr},\label{unt.10}
\end{align}
for $\xi\in \Hil$ and $r,s,t\in G$. Thus,
for $\xi\in \Hil$ and $r,s,t\in G$,
\begin{align}
v_s(\alpha_s^{**}\otimes{\mathrm{id}})(v_t)(\sigma\otimes 1)(s,t)v_{st}^*&(\xi\otimes \delta_r)\notag\\
&=v_s(\alpha_s^{**}\otimes \mathrm{id}(v_t)(\sigma \otimes 1)(s,t)(\sigma(st,t^{-1}s^{-1}r)\xi\otimes\delta_{t^{-1}s^{-1}r})\notag\\
&=v_s(\alpha_s^{**}\otimes\mathrm{id}(v_t)
(\sigma(s,t)\sigma(st,t^{-1}s^{-1}r)\xi\otimes
\delta_{t^{-1}s^{-1}r})\notag\\
&=v_s(\alpha_s^{**}(\sigma(t,t^{-1}s^{-1}r)^*)\sigma(s,t)
\sigma(st,t^{-1}s^{-1}r)\xi\otimes \delta_{s^{-1}r})\notag\\
&=\sigma(s,s^{-1}r)^*\alpha_s^{**}(\sigma(t,t^{-1}s^{-1}r)^*)
\sigma(s,t)\sigma(st,t^{-1}s^{-1}r)\xi\otimes\delta_r\notag\\
&=\xi\otimes \delta_r,\label{unt.11}
\end{align}
using \eqref{tw.2} for the last equality. Thus \eqref{tw.8} is satisfied with $\mu(s,t)=1$, and we have established exterior equivalence of the two dynamical systems. The result now
follows from Lemma \ref{exterior}.
\end{proof}

\begin{pro}\label{twmod}
Let $A$ be a unital simple C$^*$-algebra, let $G$ be a discrete group, and let $(A,G,\alpha,\sigma)$ be a twisted dynamical system, where $\alpha_g$ is outer for $g\ne e$.
If $X\subseteq A\rtimes_{\alpha,r}^\sigma G$ is a norm closed $A$-bimodule with support $S$ and
\begin{equation}\label{twmod.1}
X_S:=\ovl{\mathrm{span}}^{\|\cdot\|}\{A\lambda_\sigma(g):g\in S\},
\end{equation}
then $X_S\subseteq X$.
\end{pro}

\begin{proof}
This is identical to the proof of Proposition \ref{supp} except for one small change. Using the same notation, and referring to equation \eqref{supp.5}, the $\lambda_\sigma(g_1)$-coefficient
of $\Psi(y)$ is now
\begin{align}
&\sum_{i=1}^k a_i^*y_{g_1}\lambda_\sigma(g_1)\lambda_\sigma(g_0)^*a_i\lambda_\sigma
(g_0)\lambda_\sigma(g_1)^*\ \ \   \ \ \ (\text{from \eqref{tw.6}})\notag\\
=&\sum_{i=1}^k a_i^*y_{g_1}\sigma(g_1g_0^{-1},g_0)^*
\lambda_\sigma(g_1g_0^{-1})a_i\lambda_\sigma(g_1g_0^{-1})^*
\sigma(g_1g_0^{-1},g_0)\ \ \ \ \ (\text{from \eqref{tw.6}})\notag\\
=&\sum_{i=1}^k a_i^*y_{g_1}\sigma(g_1g_0^{-1},g_0)^*\alpha_{g_1g_0^{-1}}(a_i)
\sigma(g_1g_0^{-1},g_0).\label{twmod.2}
\end{align}
Thus the $a_i$'s are chosen so that \eqref{supp.4} is replaced by
\begin{equation}\label{twmod.3}
\left\|\sum_{i=1}^k a_i^* y_{g_1}\sigma(g_1g_0^{-1},g_0)^*\alpha_{g_1g_0^{-1}}(a_i)\right\|<\delta,
\end{equation}
which is possible by Proposition \ref{average} since $\alpha_{g_1g_0^{-1}}$ is outer and $y_{g_1}\sigma(g_1g_0^{-1},g_0)^*\in A$.
\end{proof}

We now come to the main result of this section, which characterizes the intermediate C$^*$-algebras for reduced twisted crossed products. This extends the analogous result in Theorem \ref{intermediate} for the untwisted case. We note that the conditional expectation of $A\rtimes^\sigma_{\alpha,r}G$ onto $A$ was constructed in \cite[Theorem 2.2]{Bed}.

\begin{thm}\label{maintw}
Let $A$ be a unital simple C$^*$-algebra, let $G$ be a discrete group, and let $(A,G,\alpha,\sigma)$ be a twisted dynamical system where $\alpha_g$ is outer for $g\ne e$. There is a
bijective correspondence between subgroups $H$ of $G$ and C$^*$-algebras $B$ satisfying $A\subseteq B\subseteq A\rtimes^\sigma_{\alpha,r}G$, given by
\begin{equation}\label{maintw.1}
H\mapsto A\rtimes^\sigma_{\alpha,r}H.
\end{equation}
\end{thm}

\begin{proof}
Two changes to the proof in the untwisted case (Theorem \ref{intermediate}) are necessary. The first is that Proposition \ref{twmod} is
used in place of Proposition \ref{supp}. The second is  to show that there is a $B$-continuous conditional expectation $E_H$ of
$A^{**}\rtimes^\sigma_{\alpha^{**},vn}G$ onto $A^{**}\rtimes^\sigma_{\alpha^{**},vn}H$ which, on Fourier series $x=\sum_{g\in G}x_g\lambda_\sigma(g)$, has the form
\begin{equation}\label{maintw.2}
E_H(x)=\sum_{h\in H}x_h\lambda_\sigma(h),
\end{equation}
since the proof of this fact in the untwisted case is no longer valid.

Let $\phi:(A^{**}\vntensor\B(\ell^2(G)))\rtimes^{\sigma\otimes 1}_{\alpha^{**}\otimes id,vn}G\to
(A^{**}\vntensor \B(\ell^2(G)))\rtimes_{\beta,vn}G$ be the isomorphism constructed in Lemma \ref{untwist}. As in Theorem \ref{intermediate}, there is a $B$-continuous conditional expectation
$\hat{E}_H$ of $(A^{**}\vntensor\B(\ell^2(G)))\rtimes_{\beta,vn}G$ onto $(A^{**}\vntensor\B(\ell^2(G)))\rtimes_{\beta,vn}H$, so we may define a $B$-continuous conditional expectation
\begin{equation}\label{maintw.3}
E_H: 
(A^{**}\vntensor\B(\ell^2(G)))
\rtimes^{\sigma\otimes 1}_{\alpha^{**}\otimes id,vn}G\to
(A^{**}\vntensor\B(\ell^2(G)))
\rtimes^{\sigma\otimes 1}_{\alpha^{**}\otimes id,vn}H
\end{equation} 
by $E_H=\phi^{-1}\circ \hat{E}_H\circ\phi$. Using \eqref{unt.2} and \eqref{unt.3},
and identifying $A^{**}$ with $A^{**}\otimes
1\subseteq A^{**}\vntensor \B(\ell^2(G))$,
 $E_H$ restricts to a conditional expectation of $A^{**}\rtimes^{\sigma}_{\alpha^{**},vn}G$ onto
$A^{**}\rtimes^{\sigma}_{\alpha^{**},vn}H$, satisfying \eqref{maintw.2}. Having established the existence of this conditional expectation, the proof now follows that of Theorem \ref{intermediate}.
\end{proof}

\begin{center}
\begin{tabular}{cc}
Department of Mathematics & Department of Mathematics\\
Vassar College& Texas A\&M University\\
Poughkeepsie, NY 12604 & College Station, TX 77843\\
&\\
jacameron@vassar.edu & rsmith@math.tamu.edu

\end{tabular}
\end{center}

\end{document}